\documentclass[a4paper,12pt]{article}
\usepackage{tikz}
\usetikzlibrary{arrows}
\usetikzlibrary{calc}
\usetikzlibrary{positioning}
\usetikzlibrary{cd}
\usepackage{graphicx,amsmath,amssymb,amsthm,amscd,mathrsfs}
\usepackage[margin=2.5cm]{geometry}

\newtheorem{thm}{Theorem}[section]
\newtheorem{dfn}[thm]{Definition}
\newtheorem{prp}[thm]{Proposition}

\title{A note on unfolding manifolds of meromorphic connections on the Riemann sphere with unramified singularities\\
\textit{{\normalsize{Dedicated to Prof. Yoshitsugu Takei on the occasion of his 60th birthday}}}}          %optional

\date{}
\author{Kazuki Hiroe\footnote{
	The author is supported byJSPS KAKENHI Grant Number 20K03648.}\\
Department of Mathematics and Informatics, Chiba University\\
1-33, Yayoi-cho, Inage-ku, Chiba-shi, Chiba, 263-8522 JAPAN\\
email: {\tt kazuki@math.s.chiba-u.ac.jp}
}

\begin{document}
\maketitle
\begin{abstract}      %optional
	This note explains a construction of a Poisson manifold
whose symplectic foliation describes a deformation of 
a moduli space of meromorphic connections with unramified irregular singularities.
In particular, this deformation of the moduli space corresponds 
to the unfolding of irregular singularities of the meromorphic connections.
This is an announcement of some results in the forthcoming paper.
	\end{abstract}
	\tableofcontents      %optional
	
	\section*{Introduction}
	This note is an announcement of some results in the forthcoming paper \cite{H}.
	Detailed descriptions of the statements without proofs in this note shall be 
	given in \cite{H}. 
	
	For linear ordinary equations on complex domains, 
	the confluence of singular points 
	is a classical tools to investigate irregular singularities. 
	For example, it is well-known that each differential equation of 
	Airy, Hermite-Weber, and Kummer confluent hypergeometric functions 
	is obtained from the differential equation
	\[
		z(1-z)\frac{d^{2}}{dz^{2}}w+(c-(a+b+1)z)\frac{d}{dz}w-abw=0	
	\] 
	of the Gauss hypergeometric function by the confluence of singular points 
	$z=0,1,\infty$, and further, 
	analytic properties of these functions can be related to that 
	of the Gauss hypergeometric function.
	Similarly for the Heun differential equation which has an accessory parameter and 
	is known as a generalization of 
	Gauss hypergeometric differentia equations, there is a well-known  
	family of differential equations obtained by the confluence of singular points,
	confluent Heun, bi-confluent Heun, tri-confluent Heun, and doubly-confluent Heun 
	equations.  
	Furthermore, for the study of higher dimensional Painlev\'e equations,
	Kawakami-Nakamura-Sakai constructed many confluent families of 
	differential equations with 4 accessory parameters in \cite{HKNS}.
	
	We may notice that these confluence families of differential equations 
	are considered as deformations of spaces of their accessory parameters,
	namely, we can say that the confluence of singularities 
	gives rise to deformations of moduli spaces of differential equations. 
	Indeed in \cite{CMR}, Chekhov-Mazzocco-Rubtsov defined a deformation of 
	Painlev\'e monodromy manifolds associated to the confluence of singular points 
	of differential equations. Namely, they considered 
	monodromy manifolds of 
	linear ordinary differential equations on the Riemann sphere 
	whose isomonodromic deformations give rise to Painlev\'e equations,
	and obtained explicit deformations of these manifolds which 
	correspond to the confluence of singular points 
	of differential equations.
	In case of moduli spaces of connections, 
	Inaba constructed in \cite{Ina} a one-parameter deformation of 
	 moduli spaces of meromorphic connections with unramifed irregular singular points
	 whose Hukuhara-Turrittin-Levelt normal forms have distinct eigenvalues.
	 A similar deformation also considered by Gaiur-Mazzocco-Rubtsov in 
	 \cite{GMR}, in which they defined a deformation of 
	 coadjoint orbits of a Lie algebra of polynomials which they call 
	 the Takiff algebra.  
	 Furthermore, 
	Gaiur-Mazzocco-Rubtsov and Inaba 
	considered in their papers \cite{GMR} and \cite{Ina}
	the confluence of 
	isomonodromic deformation equations.  
	
	Based on these preceding works, we shall construct a deformation space of moduli space of differential equations 
	describing the confluence of their singularities.
	Let us explain our main result. 
	We focus on the opposite operation to the confluence, i.e., 
	the unfolding of irregular singular points.
	Namely, we consider the procedure to unfold an 
	unramified irregular singular points to regular singular ones.
	Let us consider meromorphic connections on a trivial 
	bundle defined over the complex projective line $\mathbb{P}^{1}$ with unramified irregular singularities
	on a finite subset $D\subset \mathbb{P}^{1}$. 
	We fix a collection of unramified Hukuhara-Turrittin-Levelt normal forms $\mathbf{H}=(H_{a})_{a\in D}$
	at each singular point $a\in D$.
	Then Boalch introduced in \cite{Boa}
	the moduli space of irreducible meromorphic connections 
	with the fixed normal forms $\mathbf{H}$ which we denote by $\mathcal{M}_{s}^{*}(\mathbf{H})$.
	Firstly, we introduce 
	the deformation $\mathbf{H}(\mathbf{c})$, $\mathbf{c}\in \mathbb{C}^{N_{\mathbf{H}}}$, of
	the collection of normal forms $\mathbf{H}$, where $N_{\mathbf{H}}$
	is the positive integer uniquely determined by $\mathbf{H}$.   
	This is a family of unramified Hukuhara-Turrittin-Levelt normal forms 
	describing the unfolding procedure of unramified irregular singular points
	to regular singular ones.
	Then our main theorem is the following.
	\begin{thm}[Theorem \ref{thm:main1}, Theorem \ref{thm:main2}]
	Suppose that $\mathcal{M}_{s}^{*}(\mathbf{H})\neq \emptyset$.
	Then, there are a Zariski open subset $\mathbb{D}(\mathbf{H})\subset \mathbb{C}^{N_{\mathbf{H}}}$,
	holomorphic Poisson manifold $\mathcal{M}_{s}^{*}(\mathbf{H})_{\mathbb{D}(\mathbf{H})}$,
	and holomorphic map $\theta\colon \mathcal{M}_{s}^{*}(\mathbf{H})_{\mathbb{D}(\mathbf{H})}\rightarrow 
	\mathbb{D}(\mathbf{H})$ such that 
	for each $\mathbf{c}\in \mathbb{D}(\mathbf{H})$ there exits an embedding 
	\[
		\theta^{-1}(\mathbf{c})\hookrightarrow \mathcal{M}_{s}^{*}(\mathbf{H}(\mathbf{c}))	
	\]
	onto a Zariski open subset. Namely, the complex manifold  $\mathcal{M}_{s}^{*}(\mathbf{H})_{\mathbb{D}(\mathbf{H})}$
	is a deformation of the moduli space $\mathcal{M}_{s}^{*}(\mathbf{H})$ and moreover
	each fiber $\theta^{-1}(\mathbf{c})$
	is isomorphic to a Zariski open subset of the moduli space of meromorphic connections 
	$\mathcal{M}_{s}^{*}(\mathbf{H}(\mathbf{c}))$ obtained by the unfolding of $\mathbf{H}$.
	
	Furthermore, the family of the fibers $(\theta^{-1}(\mathbf{c}))_{\mathbf{c}\in \mathbb{D}(\mathbf{H})}$
	is the symplectic foliation of the regular Poisson manifold $\mathcal{M}_{s}^{*}(\mathbf{H})_{\mathbb{D}(\mathbf{H})}$
	and the above embeddings are symplectic morphisms.
	\end{thm}
	
	Our families $\mathcal{M}_{s}^{*}(\mathbf{H})_{\mathbb{D}(\mathbf{H})}$ of moduli space cover many of 
	known deformations of differential equations induced by the 
	confluence of  their singular points.
	For example, 
	confluent families explained above, i.e., families of Gauss hypergeometric equation, Heun equation, 
	and equations with 4 accessory parameters appearing in \cite{HKNS} by Kawakami-Nakamura-Sakai
	are all contained in our families 
	when we focus on unramified irregular 
	singularities.
	Our construction of the deformation is strongly influenced by 
	Oshima's work in \cite{Osh2}.
	In that paper, Oshima defined a good confluence family of linear differential equations 
	on the Riemann sphere without accessory parameter which he called 
	the versal family of rigid differential equations.
	Our result can be said to be a generalization of his versal family to non-rigid cases.
	For instance, Oshima's families recovers from our deformation spaces of moduli spaces of
	rigid connections. 
	 
	\vspace{2mm}
	\noindent
	\textbf{Acknowledgement}
	
	The author express his gratitude to Toshio Oshima for his inspiring works 
	and many variable discussions. 
	He thanks to Hiroshi Kawakami, Akane Nakamura, and Daisuke Yamakawa.
	Discussions with them in the early stage 
	help him very much to press forward this project.
	He is also grateful to Takuro Mochizuki, Hiraku Nakajima, and Vladimir Rubtsov
	for their variable comments.
	
	This article is submitted to RIMS K\^oky\^uroku Bessatsu, a Festschrift volume 
	in Honor of Yoshitsugu Takei on the occasion of his 60th birthday. The author express his configurations.
	
	\section{Moduli space of meromorphic connections on a trivial bundle over the Riemann sphere 
	with unramified irregular singularities}
	
	In the paper \cite{JMU}, Jimbo, Miwa, and Ueno developed a general study of symplectic structure of the spaces of accessory parameters of 
	differential equations with irregular singular points. After their work, in \cite{Boa}, 
	Boalch introduced a natural construction of the symplectic structure from the symplectic geometry of coadjoint orbits and the theory of 
	symplectic reduction.  
	Following Boalch's method, we give a review of the symplectic structure of the moduli spaces
	of meromorphic connections on the trivial bundle with unramified irregular singular points. 
	A detailed exposition of the contents in this section can also be found in \cite{Yam3}.
	
	\subsection{Hukuhara-Turrittin-Levelt normal forms}
	A classification theory of local and formal differential equations is obtained by Hukuhara, Turrittin, and Levelt, see \cite{Huk1}, \cite{Tur}, \cite{Lev}.
	We give a quick review of this classification theory.
	
	\begin{dfn}
	Let $R$ be either the ring of formal power series $\mathbb{C}[\![z]\!]$ or 
	the field of Laurent series $\mathbb{C}(\!(z)\!)$, and 
	$V$ be a free $R$-module of finite rank. Let us consider
	a $\mathbb{C}$-linear map
	$\nabla\colon V\rightarrow V\otimes_{R}\mathbb{C}(\!(z)\!)dz$
	satisfying the Leibniz rule,
	\[
		\nabla(fv)=v\otimes \left(df\right)+f\nabla(v),\quad
		f \in R,\ v\in V,
	\]
	which is called {\em meromorphic connection} or {\em connection} on $V$.
	Here we denote the vector space of formal meromorphic 1-forms by
	$\mathbb{C}(\!(z)\!)dz:=\left\{f(z)dz\mid f(z)\in \mathbb{C}(\!(z)\!)\right\}$,
	 and $d\colon \mathbb{C}(\!(z)\!)\rightarrow \mathbb{C}(\!(z)\!)dz$
	stands for the {\em exterior derivative} defined by
	$df(z):=\left(\frac{d}{dz}f(z)\right)dz$ for $f(z)\in R$.
	Then we also call the pair $(\nabla,V)$ a {\em formal meromorphic connection} over $R$.
	\end{dfn}
	Let $(\nabla,V)$ be a connection over $\mathbb{C}(\!(z)\!)$ of rank $n$
	and consider an algebraic field  extension $L$ of $\mathbb{C}(\!(z)\!)$.
	Then Newton-Puiseux theory tells us that there exists a positive integer $q\in \mathbb{Z}_{>0}$
	and $L\cong \mathbb{C}(\!(t)\!)$ with $t^{q}=z$.
	Then the inclusion map $\iota\colon \mathbb{C}(\!(z)\!)\ni f(z)\rightarrow f(t^{q})\in 
	\mathbb{C}(\!(t)\!)$ induces the map 
	\[
		\begin{array}{cccc}	
		\iota_{*}\colon& \mathbb{C}(\!(z)\!)dz&\longrightarrow &\mathbb{C}(\!(t)\!)dt\\
		&f(z)dz&\longmapsto& f(t^{q})qt^{q-1}\,dt
		\end{array}	
	\]
	which is compatible with the exterior derivative.
	Then we can define the connection $(\nabla_{t},V_{t})$ over $\mathbb{C}(\!(t)\!)$
	which we call the {\em extension} of $(\nabla, V)$ to $\mathbb{C}(\!(t)\!)$ as follows.
	Let $V_{t}:=\mathbb{C}(\!(t)\!)\otimes_{\mathbb{C}(\!(z)\!)} V$ be the extension of scalars
	of 
	 $V$ with 
	the inclusion map $\iota_{V}\colon V\ni v \mapsto 1\otimes v\in V_{t}$
	as $\mathbb{C}(\!(z)\!)$-vector spaces.
	As well as the above, we have the induced morphism as $\mathbb{C}(\!(z)\!)$-vector spaces,
	\[
		\begin{array}{cccc}	
		\iota_{V\,*}\colon &V\otimes_{\mathbb{C}(\!(z)\!)}\mathbb{C}(\!(z)\!)dz&\longrightarrow &
		V_{t}\otimes_{\mathbb{C}(\!(t)\!)}\mathbb{C}(\!(t)\!)dt\\
		&v\otimes f(z)dz&\longmapsto &\iota_{V}(v)\otimes \iota_{*}(f(z)dz)
		\end{array}	
	\]
	Then we define the following $\mathbb{C}$-linear map
	$\nabla_{t}\colon V_{t}\rightarrow V_{t}\otimes_{\mathbb{C}(\!(t)\!)}\mathbb{C}(\!(t)\!)dt$
	by
	\[
		\nabla_{t}(f(t)\otimes v):=\iota_{V}(v)\otimes df(t)+f(t)\otimes \iota_{V\,*}(\nabla v),\quad
		f(t)\in \mathbb{C}(\!(t)\!),\ v\in V.
	\]
	
	The classification theory of connections
	 over the algebraic closure $\mathbb{C}(\!(z)\!)^{\text{alg}}$ was developed by Hukuhara and Turrittin independently
	and improved and simplified as in the following manner by Levelt afterward.
	We say that a connection $(\nabla,V)$ over $\mathbb{C}(\!(t)\!)$ is {\em diagonalizable} when $V$ is
	a direct sum of 1-dimensional $\nabla$-invariant subspaces $V_{i}$, i.e.,  $\nabla(V_{i})\subset V_{i}\otimes \mathbb{C}(\!(t)\!)dt$.
	\begin{thm}[Hukuhara, Turrittin, Levelt, \cite{Huk1}, \cite{Tur}, \cite{Lev}]
		Let $(\nabla,V)$ be a connection over $\mathbb{C}(\!(z)\!)$.
		Then there exists an algebraic filed extension $\mathbb{C}(\!(t)\!)$ of $\mathbb{C}(\!(z)\!)$
		of degree $q$
		such that $\nabla_{t}$ has the decomposition
		\[
			t\nabla_{t}=S+N
		\]
		where $S$ is a diagonalizable connection on $V_{t}$,
		$N$ is a nilpotent $\mathbb{C}(\!(t)\!)$-linear map and
		$S$ and $N$ commute with each other, i.e., $[S, N]:=SN-NS=0$.
		The smallest integer $q$ satisfying the above is called {\em ramification index}
		of the connection.
		If the ramification index is $q=1$,
		we say that $\nabla$ is {\em unramified}.
	
	Furthermore,
		there exists a basis of $V_{t}$ over $\mathbb{C}(\!(t)\!)$
		under which
		$S$ and $N$ can be written as
		\[
			S=td_{t}-\sum_{i=0}^{k}S_{i}t^{-i}dt,\quad
			N=-N_{0}\,dt
		\]
		with diagonal matrices $S_{i}\in M_{n}(\mathbb{C})$, $i=0,1,\ldots,k$ 
		and a nilpotent matrix $N_{0}\in M_{n}(\mathbb{C})$. Here $n$ is the rank of $V$ and
		$A\,dt$ for $A\in M_{n}(\mathbb{C}(\!(t)\!))$ denotes the 
		$\mathbb{C}(\!(t)\!)$-linear map $V_{t}\ni v\mapsto Av\,dt\in 
		V_{t}\otimes_{\mathbb{C}(\!(t)\!)}\mathbb{C}(\!(t)\!)dt$.
	\end{thm}
	This leads to the following definition of normal forms in $M_{n}( \mathbb{C}(\!(t)\!))dt$
	which are usual Jordan normal forms if $k=1$. 
	\begin{dfn}[HTL normal form]\normalfont
		A {\em Hukuhara-Turrittin-Levelt normal form} or {\em HTL normal form}
		for short, is an element in $M_{n}( \mathbb{C}(\!(t)\!))dt$ of the form
		\[
			\left(\frac{S_{k}}{t^{k}}+\cdots+\frac{S_{1}}{t}+S_{0}+N_{0}\right)\frac{dt}{t}
		\]
		where $S_{i}\in M_{n}(\mathbb{C})$ are diagonal matrices commute with the nilpotent 
		matrix $N_{0}\in M_{n}(\mathbb{C})$, i.e., 
		$[S_{i},N_{0}]=0$, for all $i=0,1,\ldots,k-1$.
	The above theorem says that for a connection $(V,\nabla)$
	over $\mathbb{C}(\!(z)\!)$, we can find an extension field 
	$\mathbb{C}(\!(t)\!)$ with $t^{q}=z$ and a HTL normal form
	$H\in M_{n}(\mathbb{C}(\!(t)\!))$ such that 
	\[
		\nabla_{t}=d-H	
	\]	
	under a suitable basis of $V_{t}$ over $\mathbb{C}(\!(t)\!)$. 
	This normal form $H$ is called the HTL normal form of $(V,\nabla)$.
	In particular, we say that $(V,\nabla)$ is an {\em unramified} connection 
	if we can take an HTL normal form $H$ from $M_{n}(\mathbb{C}(\!(z)\!))dz
	\subset M_{n}(\mathbb{C}(\!(t)\!))dt$. In this case, we also call 
	this HTL normal form $H\in M_{n}(\mathbb{C}(\!(z)\!))dz$ an {unramified }
	HTL normal form.
	\end{dfn}
	
	Fix an unramified HTL normal form
	\[
		H=\left(\frac{S_{k-1}}{z^{k}}+\cdots+\frac{S_{1}}{z}+S_{0}+N_{0}\right)\frac{dz}{z}\in
		M_{n}(\mathbb{C}(\!(z)\!)dz)
	\]
	and consider an unramified meromorphic connection $(\nabla_{H},\mathbb{C}[\![z]\!]^{\oplus n})$
	over $\mathbb{C}[\![z]\!]$ defined by
	\[
		\nabla_{H}:=d-H.
	\]
	Then, a connection $(\nabla,\mathbb{C}[\![z]\!]^{\oplus n})$ is isomorphic to
	$(\nabla_{H},\mathbb{C}[\![z]\!]^{\oplus n})$ if and only if
	there exists $g(z)\in \mathrm{GL}_{n}(\mathbb{C}[\![z]\!])$
	such that
	\[
		\nabla=d-g(z)[H],
	\]
	where
	\[
		g(z)[H]:=g(z)Hg(z)^{-1}+\left(\frac{d}{dz}g(z)\right)g(z)^{-1}dz.
	\]
	Thus the isomorphic class of $(\nabla_{H},\mathbb{C}[\![z]\!]^{\oplus n})$ is parametrized by the orbit
	\[
		O_{H}:=\left\{
		g(z)[H]\in M_{n}(\mathbb{C}(\!(z)\!))dz\,\middle|\, g(z)\in \mathrm{GL}_{n}(\mathbb{C}[\![z]\!])
		\right\}.
	\]
	
	As an analogue of this isomorphic class $O_{H}$, Boalch introduced the truncated orbit of $H$
	in \cite{Boa}.
	\begin{dfn}[truncated orbit]
	 Let $H$ be an HTL normal form  
	 and we regard $H$ as an element in $M_{n}(\mathbb{C}(\!(z)\!)/\mathbb{C}[\![z]\!])dz$.
	 Then the orbit of $H$,
	 \[
		\mathbb{O}_{H}:=\left\{gHg^{-1}\in M_{n}(\mathbb{C}(\!(z)\!)/\mathbb{C}[\![z]\!])dz\,\middle|\, g\in \mathrm{GL}_{n}(\mathbb{C}[\![z]\!])\right\}
	 \]
	 is called the {\em truncated orbit} of $H$.
	\end{dfn}
	Note that the truncated orbit of $\mathbb{O}_{H}$ is characterized by $H$ in the following sense.
	\begin{prp}[Boach \cite{Boa}, Yamakawa \cite{Yam1}]
		If an HTL normal form $H'$
		contained in $\mathbb{O}_{H}$, there exists $C\in \mathrm{GL}_{n}(\mathbb{C})$ such that $H'=CHC^{-1}$.
	\end{prp}
	
	\subsection{Truncated orbits and extended orbits}
	Let a fraktur $\mathfrak{q}$ be
	a $\mathbb{C}$-subalgebra of $\mathfrak{gl}_{n}(\mathbb{C})=M_{n}(\mathbb{C})$ which can be seen 
	as a Lie subalgebra defining the bracket by $[X,Y]=XY-YX$ for $X,Y\in \mathfrak{q}$.
	Then we denote the corresponding analytic subgroup of $\mathrm{GL}_{n}(\mathbb{C})$ 
	by the roman $Q$.
	
	For a positive integer $l$, we consider a $\mathbb{C}[\![z]\!]$-module $\mathbb{C}[z^{-1}]_{l}:=z^{-(l+1)}\mathbb{C}[\![z]\!]/\mathbb{C}[\![z]\!]$.
	Since the annihilator ideal of this module is 
	$\mathrm{Ann}_{\mathbb{C}[\![z]\!]}(\mathbb{C}[z^{-1}]_{l})=\langle z^{l+1}\rangle_{\mathbb{C}[\![z]\!]}$,
	$\mathbb{C}[z^{-1}]_{l}$ can also be seen as a $\mathbb{C}[z]_{l}:=\mathbb{C}[\![z]\!]/\langle z^{l+1}\rangle_{\mathbb{C}[\![z]\!]}\cong \mathbb{C}[z]/
	\langle z^{l+1}\rangle_{\mathbb{C}[z]}$-module. 
	Usually we fix the following basis and identifications as $\mathbb{C}$-vector spaces,
	$\mathbb{C}[z^{-1}]_{l}\cong \{\sum_{i=0}^{l}\frac{a_{j}}{z^{i+1}}\mid a_{i}\in \mathbb{C}\}
	\cong \mathbb{C}^{l+1}$, $\mathbb{C}[z]_{l}\cong \{\sum_{i=0}^{l}a_{i}z^{i}\mid a_{i}\in \mathbb{C}\}\cong\mathbb{C}^{l+1}$.
	
	\subsubsection{Truncated orbits as coadjoint orbits}
	We can consider the finite dimensional complex Lie group $\mathrm{GL}_{n}(\mathbb{C}[z]_{l})$
	with the corresponding Lie algebra $\mathfrak{gl}_{n}(\mathbb{C}[z]_{l}):=M_{n}(\mathbb{C}[z]_{l})$.
	For a $\mathbb{C}$-subalgebra $\mathfrak{q}\subset \mathfrak{gl}_{n}(\mathbb{C})$,
	we define the $\mathbb{C}$-algebra $\mathfrak{q}(\mathbb{C}[z]_{l}):=\mathfrak{q}\otimes_{\mathbb{C}}\mathbb{C}[z]_{l}$
	which is naturally regarded as a a $\mathbb{C}$-subalgebra of $\mathfrak{gl}_{n}(\mathbb{C}[z]_{l})$. 
	Then we denote the corresponding analytic subgroup of e complex Lie group $\mathrm{GL}_{n}(\mathbb{C}[z]_{l})$
	by $Q(\mathbb{C}[z]_{l})$.
	
	If we regard $M_{n}(\mathbb{C}[z^{-1}]_{k})dz$ as a subspace of $M_{n}(\mathbb{C}(\!(z)\!)/
	\mathbb{C}[\![z]\!])dz$ by the inclusion map $\mathbb{C}[z^{-1}]_{k}\hookrightarrow \mathbb{C}(\!(z)\!)/\mathbb{C}[\![z]\!]$,
	then 
	the truncated orbit $\mathbb{O}_{H}$ can be seen as 
	the orbit through $H$ in $M_{n}(\mathbb{C}[z^{-1}]_{k})dz$ under the 
	action of $\mathrm{GL}_{n}(\mathbb{C}[z]_{k})$, i.e.,
	\[
		\mathbb{O}_{H}=\{gHg^{-1}\in M_{n}(\mathbb{C}[z^{-1}]_{k})dz\mid g\in \mathrm{GL}_{n}(\mathbb{C}[z]_{k})\}.
	\]
	Then the trace form 
	\[
		M_{n}(\mathbb{C}[z]_{k})\times M_{n}(\mathbb{C}[z^{-1}]_{k})dz\ni 
		(A,Bdz)\longmapsto \mathrm{res}_{z=0}(\mathrm{tr}(AB))dz\in \mathbb{C},
	\]
	is non-degenerate and thus this enable us to identify the dual space 
	$\mathfrak{gl}_{n}(\mathbb{C}[z]_{k})^{*}$ with $M_{n}(\mathbb{C}[z^{-1}]_{k})dz$.
	Thus $\mathbb{O}_{H}$
	is regarded as the coadjoint orbit of $\mathrm{GL}_{n}(\mathbb{C}[z]_{k})$
	 through $H\in M_{n}(\mathbb{C}[z^{-1}]_{k})dz
	\cong \mathfrak{gl}_{n}(\mathbb{C}[z]_{k})^{*}$.
	
	\subsubsection{Extended orbits of HTL normal forms}
	Here we recall the extended orbit which is introduced by Boalch.
	
	First we recall the fact that the coadjoint orbit of a complex Lie group $G$ through
	an element $\xi$ of $\mathfrak{g}^{*}$ is isomorphic to 
	a symplectic reduction of the cotangent bundle $T^{*}G$
	as follows.
	The cotangent bundle $\theta \colon T^{*}M\rightarrow M$ of a complex manifold $M$ 
	has the standard holomorphic symplectic form $\Omega_{T^{*}M}:=d\omega_{T^{*}M}$,
	the derivation of the canonical $1$-form $\omega_{T^{*}M}$ on 
	$T^{*}M$\footnote{Here we notice that many literature adopt $-d\omega_{T^{*}M}$ as the 
	standard symplectic form.}.
	For a holomorphic map $f\colon M\to N$ between complex manifolds, we denote the induced maps 
	$T_{m}M\rightarrow T_{f(m)}N$ and $T^{*}_{f(m)}N\rightarrow T^{*}_{m}M$
	by $(Tf)_{m}$ and $(T^{*}f)_{m}$ respectively for $m\in M$.
	Sometimes we write $f_{*,m}=(Tf)_{m}$ and $f^{*}_{m}=(T^{*}f)_{m}$ for simplicity. 
	Now we suppose $M=G$ and denote
	the right translation by $R_{g}\colon G\ni x\mapsto xg\in G$
	and the left translation by $L_{g}\colon G\ni x\mapsto gx\in G$
	for $g\in G$.
	We define the left action of $G$ on $T^{*}G$
	by 
	$\rho(g)\colon T^{*}G\ni \alpha \mapsto R^{*}_{g}(\alpha)\in T^{*}G$  
	for $g\in G$.
	Then the map 
	\[
		\mu_{G}\colon T^{*}G\ni \alpha \longmapsto L_{\theta(\alpha)}^{*}(\alpha)\in \mathfrak{g}^{*}
	\]
	is a moment map with respect to the action $\rho$, and for 
	the coadjoint orbit $\mathbb{O}_{\xi}$ through $\xi\in \mathfrak{g}^{*}$, there exists  
	an ismorphism
	\[
		\mathbb{O}_{\xi}\cong G_{\xi}\backslash \mu_{G}^{-1}(\xi)	
	\]
	as holomorphic  symplectic manifolds. Here $G_{\xi}$
	is the stabilizer subgroup of $\xi$ in $G$.
	
	Now we consider the specific coadjoint orbit $\mathbb{O}_{H}$ of $\mathrm{GL}_{n}(\mathbb{C}[z]_{k})$ through the 
	unramified HTL normal form $H$.
	Note that $\mathrm{GL}_{n}(\mathbb{C}[z]_{k})$ has the description as 
	the semi-direct product of $\mathrm{GL}_{n}(\mathbb{C})$ and the normal subgroup 
	\[
		\mathrm{GL}_{n}(\mathbb{C}[z]_{k})^{1}:=\{g(z)\in \mathrm{GL}_{n}(\mathbb{C}[z]_{k})\mid g(0)=E_{n} \},
	\]
	i.e., $\mathrm{GL}_{n}(\mathbb{C}[z]_{k})=\mathrm{GL}_{n}(\mathbb{C})\ltimes \mathrm{GL}_{n}(\mathbb{C}[z]_{k})^{1}$.
	The Lie algebra of this normal subgroup is denoted by $\mathfrak{gl}_{n}(\mathbb{C}[z]_{k})^{1}$.
	Also we denote the inclusion map of $\mathrm{GL}_{n}(\mathbb{C})$ and $\mathrm{GL}_{n}(\mathbb{C}[z]_{k})^{1}$
	into $\mathrm{GL}_{n}(\mathbb{C}[z]_{k})$ by $\iota_{0}$ and $\iota_{1}$, and 
	the projection maps from $\mathrm{GL}_{n}(\mathbb{C}[z]_{k})$ onto 
	$\mathrm{GL}_{n}(\mathbb{C})$ and $\mathrm{GL}_{n}(\mathbb{C}[z]_{k})^{1}$ by 
	$\pi_{0}$ and $\pi_{1}$ respectively.
	
	Then we can moreover consider the 
	intermediate reduction by considering the moment map with respect to the normal subgroup $\mathrm{GL}_{n}(\mathbb{C}[z]_{k})^{1}$,
	\[
		\mu_{\mathrm{GL}_{n}(\mathbb{C}[z]_{k})}^{1}:=\iota_{1}^{*}\circ \mu_{\mathrm{GL}_{n}(\mathbb{C}[z]_{k})}
		\colon  T^{*}\mathrm{GL}_{n}(\mathbb{C}[z]_{k})\rightarrow (\mathfrak{gl}_{n}(\mathbb{C}[z]_{k})^{1})^{*},
	\]
	where $\iota_{1}^{*}\colon \mathfrak{gl}_{n}(\mathbb{C}[z]_{k})^{*}\rightarrow  (\mathfrak{gl}_{n}(\mathbb{C}[z]_{k})^{1})^{*}$
	is the induced projection map by $\iota_{1}$.
	\begin{prp}
		Let us set $H_{\mathrm{irr}}:=\iota_{1}^{*}(H)$ and call it the {\em irregular type } of $H$.
		We denote the coadjoint orbit of $\mathrm{GL}_{n}(\mathbb{C}[z]_{k})^{1}$ through $H_{\mathrm{irr}}$
		by $\mathbb{O}_{H_{\mathrm{irr}}}$. Then 
		there exists an isomorphism 
		\[
			\mathrm{GL}_{n}(\mathbb{C}[z]_{k})^{1}_{H_{\mathrm{irr}}}\backslash (\mu_{\mathrm{GL}_{n}(\mathbb{C}[z]_{k})}^{1})^{-1}(H_{\mathrm{irr}})
			\cong T^{*}\mathrm{GL}_{n}(\mathbb{C})\times \mathbb{O}_{H_{\mathrm{irr}}}.
		\]
		as holomorphic symplectic manifolds.
	\end{prp}
	\begin{proof}
		This is a direct consequence of the Hamiltonian reduction by stages with respect to the semi-direct product 
		$\mathrm{GL}_{n}(\mathbb{C}[z]_{k})=\mathrm{GL}_{n}(\mathbb{C})\ltimes \mathrm{GL}_{n}(\mathbb{C}[z]_{k})^{1}$,
		see Theorem 10.5.1 in \cite{Mar}. 
	\end{proof}
	Here the isomorphism is induced from the map between the cotangent bundles
	$T^{*}(\iota_{0}\times \iota_{1})\colon T^{*}\mathrm{GL}_{n}(\mathbb{C}[z]_{k})\rightarrow T^{*}\mathrm{GL}_{n}(\mathbb{C})\times T^{*}\mathrm{GL}_{n}(\mathbb{C}[z]_{k})^{1}$
	associated to  the isomorphism given by the multiplication map
	$\iota_{0}\times \iota_{1}\colon \mathrm{GL}_{n}(\mathbb{C})\times \mathrm{GL}_{n}(\mathbb{C}[z]_{k})^{1}\rightarrow \mathrm{GL}_{n}(\mathbb{C}[z]_{k})$.
	This intermediate reduction space is called  
	the {\em extended orbit} of $H_{\mathrm{irr}}$.
	
	Let us explain that the truncated orbit $\mathbb{O}_{H}$ is a symplectic reduction of
	this extended orbit. 
	Since $\mathrm{GL}_{n}(\mathbb{C}[z]_{k})^{1}$ is a normal subgroup of $\mathrm{GL}_{n}(\mathbb{C}[z]_{k})$,
	the subgroup $\mathrm{GL}_{n}(\mathbb{C})\subset \mathrm{GL}_{n}(\mathbb{C}[z]_{k})$ acts on $\mathrm{GL}_{n}(\mathbb{C}[z]_{k})^{1}$
	by the adjoint action.
	Thus we can consider the induced action of $\mathrm{GL}_{n}(\mathbb{C})$ on $(\mathfrak{gl}_{n}(\mathbb{C}[z]_{k})^{1})^{*}$.
	We denote the stabilizer subgroup of $H_{\mathrm{irr}}\in (\mathfrak{gl}_{n}(\mathbb{C}[z]_{k})^{1})^{*}$
	in $\mathrm{GL}_{n}(\mathbb{C})$ by $\mathrm{GL}_{n}(\mathbb{C})_{H_{\mathrm{irr}}}$
	and put $\iota_{0,H_\mathrm{irr}}\colon \mathrm{GL}_{n}(\mathbb{C})_{H_{\mathrm{irr}}}\hookrightarrow \mathrm{GL}_{n}(\mathbb{C}[z]_{k})$,
	the inclusion map.
	Then we can define a map from $T^{*}\mathrm{GL}_{n}(\mathbb{C}[z]_{k})^{1}$ to $\mathfrak{gl}_{n}(\mathbb{C})_{H_{\mathrm{irr}}}^{*}$,
	the dual space of the Lie algebra $\mathfrak{gl}_{n}(\mathbb{C})_{H_{\mathrm{irr}}}$ of $\mathrm{GL}_{n}(\mathbb{C})_{H_{\mathrm{irr}}}$,
	as the composite map 
	\begin{multline*}
	\mu_{\mathrm{GL}_{n}(\mathbb{C}[z]_{k})^{1}}^{0}
	\colon T^{*}\mathrm{GL}_{n}(\mathbb{C}[z]_{k})^{1}\xrightarrow{T^{*}\pi_{1}} T^{*} \mathrm{GL}_{n}(\mathbb{C}[z]_{k})\\
	\xrightarrow{\mu_{\mathrm{GL}_{n}(\mathbb{C}[z]_{k})}}\mathfrak{gl}_{n}(\mathbb{C}[z]_{k})^{*}
	\xrightarrow{\iota_{0,H_\mathrm{irr}}^{*}}\mathfrak{gl}_{n}(\mathbb{C})_{H_{\mathrm{irr}}}^{*}.
	\end{multline*}
	Here $T^{*}\pi_{1}\colon T^{*}\mathrm{GL}_{n}(\mathbb{C}[z]_{k})^{1}\rightarrow  T^{*} \mathrm{GL}_{n}(\mathbb{C}[z]_{k})$
	is defined by $T^{*}\pi_{1}(\alpha):=(T^{*}\pi_{1})_{\iota_{1}(\theta(\alpha))}(\alpha)$ for $\alpha\in T^{*}\mathrm{GL}_{n}(\mathbb{C}[z]_{k})^{1}$.
	Then we can show the following.
	\begin{prp}
		The restriction of $\mu_{\mathrm{GL}_{n}(\mathbb{C}[z]_{k})^{1}}^{0}$
		on $\mu_{\mathrm{GL}_{n}(\mathbb{C}[z]_{k})^{1}}^{-1}(H_{\mathrm{irr}})$ factors 
		through the quotient map 
		\[
			q\colon \mu_{\mathrm{GL}_{n}(\mathbb{C}[z]_{k})^{1}}^{-1}(H_{\mathrm{irr}})\rightarrow 
		\mathbb{O}_{\mathrm{irr}}=\mathrm{GL}_{n}(\mathbb{C}[z]_{k})^{1}_{H_{\mathrm{irr}}}\backslash \mu_{\mathrm{GL}_{n}(\mathbb{C}[z]_{k})^{1}}^{-1}(H_{\mathrm{irr}}).
		\]
		Namely, there uniquely exists the map $\mathrm{res}_{\mathbb{O}_{H_{\mathrm{irr}}}}\colon 
		\mathbb{O}_{H_{\mathrm{irr}}}\rightarrow \mathfrak{gl}_{n}(\mathbb{C})_{H_{\mathrm{irr}}}^{*}$
		such that  the diagram 
		\[
		\begin{tikzcd}
			\mu_{\mathrm{GL}_{n}(\mathbb{C}[z]_{k})^{1}}^{-1}(H_{\mathrm{irr}}) \arrow[r,"\mu_{\mathrm{GL}_{n}(\mathbb{C}[z]_{k})^{1}}^{0}"] \arrow[d,"q"]&
			[3 em]\mathfrak{gl}_{n}(\mathbb{C})^{*}_{H_{\mathrm{irr}}}\\
			\mathbb{O}_{H_{\mathrm{irr}}}\arrow[ur, "\mathrm{res}_{\mathbb{O}_{H_{\mathrm{irr}}}}"']&
		\end{tikzcd}
	\]
	is commutative.
	\end{prp}
	\begin{proof}
		Since the restriction map $\mu_{\mathrm{GL}_{n}(\mathbb{C}[z]_{k})^{1}}^{0}|_{\mu_{\mathrm{GL}_{n}(\mathbb{C}[z]_{k})^{1}}^{-1}(H_{\mathrm{irr}})}$
		is $\mathrm{GL}_{n}(\mathbb{C}[z]_{k})^{1}_{H_{\mathrm{irr}}}$-invariant,
		this map factors through the quotient map.
	\end{proof}
	Let $\iota_{H_{\mathrm{irr}}}\colon \mathrm{GL}_{n}(\mathbb{C})_{H_{\mathrm{irr}}}\hookrightarrow \mathrm{GL}_{n}(\mathbb{C})$
	be the inclusion map. Then Boalch showed that the truncated orbit $\mathbb{O}_{H}$ is a symplectic 
	reduction of the extended orbit $T^{*}\mathrm{GL}_{n}(\mathbb{C})\times \mathbb{O}_{H_{\mathrm{irr}}}$ as follows.
	\begin{thm}[Boalch \cite{Boa}]\label{thm:boalch}
		Let us consider the map 
		\[
			\begin{array}{cccc}	
		\mu_{\mathrm{ext}}\colon &	T^{*}\mathrm{GL}_{n}(\mathbb{C})\times \mathbb{O}_{H_{\mathrm{irr}}}&\longrightarrow 
			&\mathfrak{gl}_{n}(\mathbb{C})_{H_{\mathrm{irr}}}^{*}\\
			&(\alpha,\xi)&\longmapsto &\iota_{H_{\mathrm{irr}}}^{*}\circ \mu_{\mathrm{GL}_{n}(\mathbb{C})}(\alpha)+\mathrm{res}_{\mathbb{O}_{H_{\mathrm{irr}}}}(\xi)
			\end{array}.
		\]
		Then this map is a moment map with respect to the $\mathrm{GL}_{n}(\mathbb{C})_{H_{\mathrm{irr}}}$-action 
		which is induced from the natural $\mathrm{GL}_{n}(\mathbb{C})$-action on $T^{*}\mathrm{GL}_{n}(\mathbb{C})\times T^{*}\mathrm{GL}_{n}(\mathbb{C}[z]_{k})^{1}
		\cong T^{*}\mathrm{GL}_{n}(\mathbb{C}[z]_{k})$. Moreover setting $H_{\mathrm{res}}:=\iota_{0}^{*}(H)$,
		we have a symplectic isomorphism
		\[
			\mathbb{O}_{H}\cong (\mathrm{GL}_{n}(\mathbb{C})_{H_{\mathrm{irr}}})_{H_{\mathrm{res}}}\backslash \mu_{\mathrm{ext}}^{-1}(H_{\mathrm{res}}).
		\]
	\end{thm}
	
	\subsection{Triangular factorization of $\mathbb{O}_{H_{\mathrm{irr}}}$}
	We explain
	the coadjoint orbit $\mathbb{O}_{H_{\mathrm{irr}}}$ has a nice affine coordinate system 
	which is very important to define a deformation of the truncated orbit in the latter section.
	
	For a subgroup $G\subset \mathrm{GL}_{n}(\mathbb{C}[z]_{l})$, we denote the 
	intersection $G\cap \mathrm{GL}_{n}(\mathbb{C}[z]_{l})^{1}$ by $G^{1}$.
	Also we denote the intersection $\mathfrak{gl}_{n}(\mathbb{C}[z]_{l})^{1}$ with 
	a Lie subalgebra $\mathfrak{g}\subset \mathfrak{gl}_{n}(\mathbb{C}[z]_{l})$ by $\mathfrak{g}^{1}$.
	
	\subsubsection{Compositions of a natural number and subgroups of $\mathrm{GL}_{n}(\mathbb{C})$}
	\label{sec:part}
	Let us recall compositions of natural numbers.
	Let $n$ and $k$ be strictly positive integers. Then we denote
	the set of {\em weak composition} of $n$ with $l$ terms by 
	\[
		\mathcal{C}(n,l):=\left\{
	(a_{1},a_{2},\ldots,a_{l})\in \mathbb{Z}_{\ge 0}^{l}\,\middle|\,
	a_{1}+a_{2}+\cdots+a_{l}=n
	\right\}
	\]
	where we allow some components $a_{i}$ to be zero.
	If all components in $(a_{1},a_{2},\ldots,a_{l})\in \mathcal{C}(n,l)$ are strictly positive integers, then it is called 
	{\em strict composition} or {\em composition} shortly and 
	we denote the set of strict composition of $n$ with $l$ terms 
	by $\mathcal{C}^{+}(n,l).$
	The sets of all weak compositions and strict compositions of $n$ are 
	denoted by
	\[
		\mathcal{C}(n):=\bigsqcup_{l\in \mathbb{Z}_{>0}}\mathcal{C}(n,l),\quad\quad
		\mathcal{C}^{+}(n):=\bigsqcup_{l\in \mathbb{Z}_{>0}}\mathcal{C}^{+}(n,l),
	\]
	respectively.
	
	Let us take $\mathbf{m}=\{n_{1},n_{2},\ldots,n_{m}\}\in 
	\mathcal{C}(n)$. Let $\mathbf{e}_{i}\in \mathbb{C}^{n}$ be the vector 
	with $1$ as the $i$-th entry and $0$ as the others for each $i=1,2,\ldots,n.$
	Then we define 
	inclusion maps
	\[
		\begin{array}{cccc}
			\iota_{i}\colon &\mathbb{C}^{n_{i}}&\longrightarrow &\mathbb{C}^{n}\\
				&(a_{1},a_{2},\ldots,a_{n_{i}})&\longmapsto&
				\displaystyle \sum_{j=n_{1}+n_{2}+\cdots +n_{i-1}+1}^{n_{1}+n_{2}+\cdots +n_{i}}a_{j}\mathbf{e}_{j}
		\end{array}
	\]
	for $i=1,2,\ldots,m$ and an isomorphism 
	\[
		\begin{array}{cccc}
			\iota_{\mathbf{m}}\colon &\bigoplus_{i=1}^{m}\mathbb{C}^{n_{i}}&\longrightarrow &\mathbb{C}^{n}\\
			&(\mathbf{a}_{i})_{i=1,2,\ldots,m}&\longmapsto&
			\sum_{i=1}^{m}\iota_{i}(\mathbf{a}_{i})
		\end{array},
	\]
	where if $n_{i}=0$, then we set $\iota_{i}$ as the zero map.
	We identify $\mathbb{C}^{n}$ and $\bigoplus_{i=1}^{m}\mathbb{C}^{n_{i}}$ through this isomorphism.
	
	Let us regard $\mathrm{Hom}_{\mathbb{C}}(\mathbb{C}^{n_{i}},\mathbb{C}^{n_{i'}})$ as a subspace
	of $\mathrm{End}_{\mathbb{C}}(\mathbb{C}^{n})=M_{n}(\mathbb{C})$ through 
	the canonical projection $\mathrm{pr}_{i}\colon \mathbb{C}^{n}\rightarrow \mathbb{C}^{n_{i}}$
	and injection $\iota_{i'}\colon \mathbb{C}^{n_{i'}}\rightarrow \mathbb{C}^{n}$.
	Then we define the following $\mathbb{C}$-subalgebras of $M_{n}(\mathbb{C})$, 
	\begin{align*}
		\mathfrak{l}_{\mathbf{m}}&:=\bigoplus_{i\in \{1,2,\ldots,m\}}
		\mathrm{Hom}_{\mathbb{C}}(\mathbb{C}^{n_{i}},\mathbb{C}^{n_{i}}),\\
		\mathfrak{n}^{+}_{\mathbf{m}}&:=\bigoplus_{\substack{i,i'\in \{1,2,\ldots,m\}\\ i< i'}}
		\mathrm{Hom}_{\mathbb{C}}(\mathbb{C}^{n_{i}},\mathbb{C}^{n_{i'}}),&
		\mathfrak{n}^{-}_{\mathbf{m}}&:=\bigoplus_{\substack{i,i'\in \{1,2,\ldots,m\}\\i > i'}}
		\mathrm{Hom}_{\mathbb{C}}(\mathbb{C}^{n_{i}},\mathbb{C}^{n_{i'}}),
	\end{align*}
	and regard them as the Lie subalgebras of $\mathfrak{gl}_{n}(\mathbb{C})=M_{n}(\mathbb{C})$.
	Let 
	 $L_{\mathbf{m}}$, $N^{+}_{\mathbf{m}}$, 
	and $N^{-}_{\mathbf{m}}$ be the corresponding 
	analytic subgroups of $\mathrm{GL}_{n}(\mathbb{C})$.

	\subsubsection{Sequences of compositions of positive integers}\label{sec:seqence}
	Let us take weak compositions $\alpha=(\alpha_{1},\alpha_{2},\ldots,\alpha_{l}),\,\beta=(\beta_{1},\beta_{2},\ldots,\beta_{m})\in \mathcal{C}(n)$.
	Then we say that $\beta$ is a {\em refinement} of $\alpha$ if there exists a surjection $\phi\colon 
	\{1,2,\ldots,m\}\rightarrow \{1,2,\ldots,l\}$ from the index set of $\beta$ to that of $\alpha$
	such that equations $\alpha_{i}=\sum_{j\in \phi^{-1}(i)}\beta_{j}$ hold for all $i=1,2,\ldots,l$.
	In this case we write $\beta\le \alpha$ or 
	\[
		\beta\le_{\phi}\alpha,	
	\]
	emphasizing the surjection $\phi$ which we call the {\em refinement map} between $\alpha$ and $\beta$.
	Let $\widetilde{\mathbf{m}}=(\mathbf{m}_{1},\mathbf{m}_{2},\ldots,\mathbf{m}_{k})$ be 
	an ascending sequence of compositions of $n$, i.e., entries of $\widetilde{\mathbf{m}}$ consist of the following 
	ascending sequence of refinement,
	\[
		\mathbf{m}_{1}\le_{\phi_{1}}\mathbf{m}_{2}\le_{\phi_{2}}
		\cdots\le_{\phi_{k-1}}\mathbf{m}_{k}.
	\]
	We write $\mathbf{m}_{i}=\{n_{\langle i,1\rangle},n_{\langle i,2\rangle},\ldots,n_{\langle i,m(i)\rangle}\}$
	and define a total ordering $\preceq_{i}$ on the index set $\{\langle i,1\rangle,\langle i,2\rangle,\ldots,\langle i,m(i)\rangle\}$ 
	of $\mathbf{m}_{i}$ for each $i=1,2,\ldots,k$.
	We say that the collection of orderings $\{\preceq_{i}\}_{i=1,2,\ldots,k}$ is the {\em ordering of the sequence $\widetilde{\mathbf{m}}$} if 
	\[
		\langle	i,j\rangle\preceq_{i}\langle i,j'\rangle\ \text{ imply }\ \phi_{i}(\langle	i,j\rangle)\preceq_{i+1}\phi_{i}(\langle	i,j'\rangle)
	\] 
	for any $\langle	i,j\rangle,\langle i,j'\rangle\in \{\langle i,1\rangle,\langle i,2\rangle,\ldots,\langle i,m(i)\rangle\}$
	and $i=1,2,\ldots,k-1$.
	
	Let us fix a surjection $\phi^{1}\colon \{1,2,\ldots,n\}\rightarrow \{\langle 1,1\rangle,\,\langle 1,2\rangle\,
	\ldots,\,\langle 1,m(1)\rangle\}$ and 
	define the maps $\phi^{i}\colon \{1,2,\ldots,n\}\rightarrow \{\langle i,1\rangle,\,\langle i,2\rangle\,
	\ldots,\,\langle i,m(i)\rangle\}$ by 
	$\phi^{i}:=\phi_{i-1}\circ \phi_{i-2}\circ \cdots \circ\phi_{1}\circ \phi^{1}$ for $i=2,3,\ldots,k$.
	Then we can 
	regard the vector space $\mathbb{C}^{n_{\langle i,j\rangle}}$
	as the subspace of $\mathbb{C}^{n}$ by the equation
	\[
		\mathbb{C}^{n_{\langle i,j\rangle}}=\bigoplus_{k\in (\phi^{i})^{-1}(j)}\mathbb{C}\mathbf{e}_{k},	
	\]
	for each $i=1,2,\ldots, k$ and $j=1,2,\ldots, m(i)$. 
	Then we obtain the decompositions 
	\[
		\mathbb{C}^{n}=	\bigoplus_{j=1}^{m(i)}\mathbb{C}^{n_{\langle i,j\rangle}}
	\]
	for $i=1,2,\ldots,k$ associated to $\widetilde{\mathbf{m}}$.
	
	As we saw in the previous section, a composition of $n$ defines the block decomposition of $\mathfrak{gl}_{n}(\mathbb{C})=M_{n}(\mathbb{C})$.
	Therefore a pair of a sequence $\widetilde{\mathbf{m}}$ of compositions and an ordering $\{\preceq_{i}\}_{i=1,2,\ldots,k}$ of $\widetilde{\mathbf{m}}$ defines 
	the following Lie subalgebras of $\mathfrak{gl}_{n}(\mathbb{C})$ with 
	inclusion relations
	\begin{equation}\label{eq:inclusion}
	\begin{split}
		&\mathfrak{l}_{\mathbf{m}_1}\subset \mathfrak{l}_{\mathbf{m}_2}\subset \cdots\subset \mathfrak{l}_{\mathbf{m}_{k}},\\
		&\mathfrak{n}_{\mathbf{m}_{k}}^{\pm}\subset \cdots\subset \mathfrak{n}_{\mathbf{m}_2}^{\pm}\subset \mathfrak{n}_{\mathbf{m}_1}^{\pm}.
	\end{split}
	\end{equation}
	Moreover 
	we define 
	\[
		\mathfrak{n}^{\pm}_{\mathbf{m}_{i}^{i'}}:=\mathfrak{l}_{\mathbf{m}_{i'}}\cap \mathfrak{n}_{\mathbf{m}_{i}},	
	\]
	for pairs $(i,i')$ satisfying $1\le i<i'\le k+1$,
	which give us further decompositions 
	\[
		\mathfrak{l}_{\mathbf{m}_{i'}}=\mathfrak{n}^{-}_{\mathbf{m}_{i}^{i'}}\oplus\mathfrak{l}_{\mathbf{m}_{i}}
		\oplus	\mathfrak{n}^{+}_{\mathbf{m}_{i}^{i'}}.
	\]
	Here we formally put $\mathfrak{l}_{\mathbf{m}_{k+1}}=\mathfrak{gl}_{n}(\mathbb{C})$.
	Throughout this note, when we consider a sequence $\widetilde{\mathbf{m}}$ of compositions,
	we always fix an ordering $\{\preceq_{i}\}_{i=1,2,\ldots,k}$ of $\widetilde{\mathbf{m}}$.
	
	\subsubsection{Spectral types of HTL-normal forms}\label{sec:spec}
	We shall introduce spectral types of HTL normal forms.
	As the easiest case, let us first consider a square matrix
	$X\in M_{n}(\mathbb{C})$ and 
	define the spectral type of $X$ as follows.
	Take the Jordan decomposition of $X$,
	\[
		X=S+N
	\]
	with the diagonalizable matrix $S$ and the nilpotent matrix $N$ satisfying the relation $[S,N]=0$.
	Then we have the decomposition $\mathbb{C}^{n}=\bigoplus_{i=1}^{m}V_{i}$
	as the direct sum of eigenspaces of $S$ which defines a composition of $n$,
	$$\mathbf{m}:=(\mathrm{dim}_{\mathbb{C}}V_{1},\mathrm{dim}_{\mathbb{C}}V_{2},\ldots,\mathrm{dim}_{\mathbb{C}}V_{m}).$$
	Let us look at the nilpotent matrix $N$ and notice that
	the restriction $N_{i}:=N|_{V_{i}}$ on $V_{i}$ defines a nilpotent
	endomorphism of $V_{i}$ for each $i=1,2,\ldots,m$ 
	since $[S,N]=0$.
	Then we have the partial flag $\mathcal{F}_{N;V_{i}}$,
	\[
		V_{i}\supsetneqq (V_{i})_{1}:=\mathrm{Im}N_{i}\supsetneqq (V_{i})_{2}:=\mathrm{Im}N_{i}^{2}\supsetneqq
		\cdots\supsetneqq (V_{i})_{d_{i}}:=	\mathrm{Im}N_{i}^{d_{i}}=\{0\},
	\]
	and descending sequence of positive integers
	\[
		\mathrm{dim}_{\mathbb{C}}V_{i}> \mathrm{dim}_{\mathbb{C}}(V_{i})_{1}> \mathrm{dim}_{\mathbb{C}}(V_{i})_{2}>
		\cdots > \mathrm{dim}_{\mathbb{C}}(V_{i})_{d_{i}}=0.
	\]
	We call the sequence of the integers 
	\[
		\sigma(\mathcal{F}_{N;V_{i}}):=(\mathrm{dim}_{\mathbb{C}}V_{i}, \mathrm{dim}_{\mathbb{C}}(V_{i})_{1}, \mathrm{dim}_{\mathbb{C}}(V_{i})_{2},
	\ldots,\mathrm{dim}_{\mathbb{C}}(V_{i})_{d_{i}-1})	
	\]
	the {\em signature} of $\mathcal{F}_{N;V_{i}}$.
	 We denote the set of the signatures by 
	 \[
		\sigma({\mathbf{m}}):=\{\sigma(\mathcal{F}_{N;V_{i}})\}_{i=1,2,\ldots,m} 
	 \]
	 and call it the {\em signature} of $\mathbf{m}$.

	\begin{dfn}[spectral type of a square matrix]
		Let $X\in M_{n}(\mathbb{C})$ be a square matrix with the
		Jordan decomposition $X=S+N$.
		Then we can define the composition $\mathbf{m}$ of $n$ and the signature $\sigma(\mathbf{m})$ of $\mathbf{m}$ as above.
		we call the pair
		\[
			(\mathbf{m},\sigma(\mathbf{m}))	
		\]
		the {\em spectral type} of $X$.
		In particular, if $X=S$, i.e., $X$ is a semisimple matrix, then the signature
		$\sigma(\mathbf{m})$ is trivial. In this case, we  
		simply denote the spectral type by 
		$(\mathbf{m},\mathrm{triv}).$
	\end{dfn}
	
	Next we consider an unramified HTL-normal form $\displaystyle
		\left(\frac{S_{k}}{z^{k}}+\cdots+\frac{S_{1}}{z}+S_{0}+N_{0}\right)\frac{dz}{z}.$
	Then we obtain the decomposition $\mathbb{C}^{n}=\bigoplus_{j=1}^{m(i)}V_{\langle i;j\rangle}$
	as the direct sum of simultaneous eigenspaces of $(S_{i},S_{i+1},\ldots,S_{k})$
	which defines the following composition of $n$,
	\[
		\mathbf{m}_{i}:=(\mathrm{dim\,}_{\mathbb{C}}V_{\langle i;1 \rangle},\mathrm{dim\,}_{\mathbb{C}}V_{\langle i;2\rangle},
		\ldots,\mathrm{dim\,}_{\mathbb{C}}V_{\langle i;m(i)\rangle}),
	\]
	for each $i=0,1,\ldots,k$.
	Also for each  $i=0,1,\ldots,k-1$, let us define the map 
	\[
	\phi_{i}\colon\{\langle i,1\rangle ,\langle i,2\rangle,\ldots,\langle i,m(i)\rangle \}\rightarrow 
	\{\langle i+1,1\rangle,\langle i+1,2\rangle,\ldots,\langle i+1,m(i+1)\rangle\}
	\]
	between index sets of $\mathbf{m}_{i}$ and $\mathbf{m}_{i+1}$ which satisfies
	that
	\[
		V_{\langle i+1;j\rangle}=\bigoplus_{\mu\in\phi^{-1}_{i}(j)}
		V_{\langle i;\mu\rangle}
	\]
	for all $j\in \{1,2,\ldots,m(i)\}$.
	Then we obtain a refinement sequence of partitions
	\[
		\mathbf{m}_{0}\le_{\phi_{0}}\le  \mathbf{m}_{1}\le_{\phi_{1}} \cdots \le_{\phi_{k-1}}\mathbf{m}_{k}.
	\]
	
	The restrictions of nilpotent matrices 
	$N_{j}:=N_{0}|_{V_{\langle 0;j\rangle}}\in \mathrm{End}_{\mathbb{C}}(V_{\langle 0,j\rangle})$
	define the flags $\mathcal{F}(N_{0},V_{\langle 0;j\rangle})$ of $V_{\langle 0;j\rangle}$,
	\[
		V_{\langle 0;j\rangle}\supsetneqq (V_{\langle 0;j\rangle})_{1}:=
		\mathrm{Im}N_{j}\supsetneqq 
		(V_{\langle 0;j\rangle})_{2}:=
		\mathrm{Im}N_{j}^{2}\supsetneqq
		\cdots \supsetneqq 
		(V_{\langle 0;j\rangle})_{d_{j}}
		:=
		\mathrm{Im}N_{j}^{d_{j}}=\{0\},
	\]
	with the signature
	\[
		\sigma(\mathcal{F}(N_{0},V_{\langle 0;j\rangle}))=
		(\mathrm{dim}_{\mathbb{C}}V_{\langle 0;j\rangle},\mathrm{dim}_{\mathbb{C}}(V_{\langle 0;j\rangle})_{1},
		\mathrm{dim}_{\mathbb{C}}(V_{\langle 0;j\rangle})_{2},
		\ldots,
		\mathrm{dim}_{\mathbb{C}}(V_{\langle 0;j\rangle})_{d_{j}-1}).
	\]
	Also we set 
	\[
		\sigma(\mathbf{m}_{0}):=(\sigma(\mathcal{F}(N_{0},V_{\langle 0;j\rangle})))_{j=1,2,\ldots,m(0)}.	
	\]

	\begin{dfn}[spectral types of HTL normal forms]\normalfont
		The pair 
		\[
			\mathrm{sp}(H):=(\mathbf{m}_{0}\le_{\phi_{0}}\le  \mathbf{m}_{1}\le_{\phi_{1}} \cdots \le_{\phi_{k-1}}\mathbf{m}_{k}, \sigma(\mathbf{m}_{0}))
		\]
		of the 
		refinement sequence of partitions of $n$ and 
		 signature $\sigma(\mathbf{m}_{0})$ of $\mathbf{m}_{0}$ defined above
		 is called the {\em spectral type} of the HTL-normal form
		\[
			H=\left(\frac{S_{k}}{z^{k}}+\cdots+\frac{S_{1}}{z}+S_{0}+N_{0}\right)\frac{dz}{z}.
		\]
		When $N_{0}=0$, the signature $\sigma(\mathbf{m}_{0})$ is trivial.
		Thus we denote the spectral type by 
		\[
			(\mathbf{m}_{0}\le_{\phi_{0}}\le  \mathbf{m}_{1}\le_{\phi_{1}} \cdots \le_{\phi_{k-1}}\mathbf{m}_{k}, \mathrm{triv})
		\]
		in this case.
	\end{dfn}
	
	\subsubsection{Triangular factrization of $\mathbb{O}_{H_{\mathrm{irr}}}$}
	Firstly, we recall a general fact about coadjoint orbits.
	It is well-known that a coadjoint orbit 
	$\mathbb{O}_{\xi}$ of a complex Lie group $G$ through $\xi 
	\in \mathfrak{g}^{*}$ is a immersed submanifold
	of $\mathfrak{g}^{*}$ via the injective immersion 
	\[
		\mathbb{O}_{\xi}\cong G/G_{\xi}\ni [g]\longmapsto \mathrm{Ad}^{*}(g)(\xi)\in \mathfrak{g}^{*}, 	
	\]
	where $[g]\in G/G_{\xi}$ is the class of $g\in G$.
	Under the identification $\mathbb{O}_{\xi}\cong G_{\xi}\backslash \mu_{G}^{-1}(\xi)$,
	this injective immersion is written as follows.
	Let us consider the map 
	\[
		\nu_{G}\colon T^{*}G\ni \alpha \longmapsto R^{*}_{\theta(\alpha)}(\alpha)\in \mathfrak{g}^{*},	
	\]
	where $\theta\colon  T^{*}G\rightarrow G$ is the natural projection.
	Then we note that this map is invariant under
	the action $\rho$ of $G$ on $T^{*}G$, i.e., we have $\nu_{G}\circ \rho(g)=\nu_{G}$
	for all $g\in G$. Therefore the restriction map
	$\nu_{G}\colon \mu_{G}^{-1}(\xi)\rightarrow \mathfrak{g}^{*}$ factors through 
	the quotient map $q\colon \mu_{G}^{-1}(\xi)\rightarrow \mathbb{O}_{\xi}$
	and there uniquely exists the map $\mathbb{O}_{\xi}\rightarrow \mathfrak{g}^{*}$ such 
	that the diagram 
	\[
		\begin{tikzcd}
			\mu_{G}^{-1}(\xi) \arrow[r,"\nu_{G}"] \arrow[d,"q"]&
			\mathfrak{g}^{*}\\
			\mathbb{O}_{\xi}\arrow[ur, ]&
		\end{tikzcd}
	\]
	is commutative. Then we can check that this map $\mathbb{O}_{\xi}\rightarrow \mathfrak{g}^{*}$ coincides with 
	the above injective immersion.
	
	Let us consider the subgroups $(N^{-})^{(l)}:=
	N_{\mathbf{m}_{l}^{l+1}}^{-}(\mathbb{C}[z]_{l-1})^{1}\subset 
	\mathrm{GL}_{n}(\mathbb{C}[z]_{l-1})^{1}$ for $l=2,3,\ldots,k$.
	Then we recall that $\mathrm{GL}_{n}(\mathbb{C}[z]_{l-1})^{1}$ admits the LU decompositions.
	\begin{prp}[LU decomposition]
		Let $\mathbf{m}$ be a composition of $n$. Then for a positive integer $j$,
		the multiplication gives the isomorphism  as complex manifolds,
		\[
			N_{\mathbf{m}}^{-}(\mathbb{C}[z]_{j})^{1}\times L_{\mathbf{m}}(\mathbb{C}[z]_{j})\times 	N_{\mathbf{m}}^{+}(\mathbb{C}[z]_{j})^{1}
			\ni (n^{-},h,n^{+})\longmapsto n^{-}\cdot h\cdot n^{+}\in \mathrm{GL}_{n}(\mathbb{C}[z]_{j})^{1}.
		\]
	\end{prp}
	\begin{proof}
		Since any principal minors of $g\in \mathrm{GL}_{n}(\mathbb{C}[z]_{j})^{1}$ 
		are units in $\mathbb{C}[z]_{j}$, thus 
		the usual LU decomposition algorithm is valid in this setting.
	\end{proof}
	Thus we obtain 
	the projection
	$\mathrm{pr}_{(N^{-})^{(l)}}\colon \mathrm{GL}_{n}(\mathbb{C}[z]_{l-1})^{1}\rightarrow (N^{-})^{(l)}$
	along the LU decomposition. 
	Regarding the projection $\mathrm{pr}^{k}_{l-1}\colon \mathbb{C}[z]_{k}\rightarrow \mathbb{C}[z]_{l-1}$
	as a map of $\mathbb{C}$-vector spaces,
	we can define a section 
	$s^{k}_{l-1}\colon \mathbb{C}[z]_{l-1}\ni \sum_{i=0}^{l-1}c_{i}z^{i}\mapsto \sum_{i=0}^{l-1}c_{i}z^{i}
	\in \mathbb{C}[z]_{k}$.
	Then we define the projection 
	\[
		\pi_{(N^{-})^{(l)}}:=\mathrm{pr}_{(N^{-})^{(l)}}\circ \mathrm{pr}^{k}_{l-1}
		\colon \mathrm{GL}_{n}(\mathbb{C}[z]_{k})^{1}\rightarrow (N^{-})^{(l)}.
	\]
	Also 
	the inclusion map $i_{(N^{-})^{(l)}}\colon (N^{-})^{(l)}
	\hookrightarrow \mathrm{GL}_{n}(\mathbb{C}[z]_{l})^{1}$ defines a section 
	\[
		\iota_{_{(N^{-})^{(l)}}}:=s^{k}_{l-1}\circ i_{(N^{-})^{(l)}}\colon
		(N^{-})^{(l)}
	\hookrightarrow \mathrm{GL}_{n}(\mathbb{C}[z]_{k})^{1} 
	\]	
	of the projection $\pi_{(N^{-})^{(l)}}$. 
	By using these projection and section, we obtain a map between cotangent bundles,
	\[
		T^{*}\pi_{(N^{-})^{(l)}}\colon T^{*}(N^{-})^{(l)}\ni 
		\alpha\mapsto (T^{*}\pi_{1,(N^{-})^{(l)}})_{\iota_{(N^{-})^{(l)}}(\theta(\alpha))}(\alpha)
		\in T^{*}\mathrm{GL}_{n}(\mathbb{C}[z]_{k})^{1}.
	\]
	
	For $H_{\mathrm{irr}}\in (\mathfrak{gl}_{n}(\mathbb{C}[z]_{k})^{1})^{*}$, we 
	set $\omega_{H_{\mathrm{irr}}}:=(L^{*}_{g^{-1}}(H_{\mathrm{irr}}))_{g\in \mathrm{GL}_{n}(\mathbb{C}[z]_{k})^{1}}$,
	the left invariant $1$-form on $\mathrm{GL}_{n}(\mathbb{C}[z]_{k})^{1}$
	associated to $H_{\mathrm{irr}}$.
	Define a map $F\colon \prod_{l=2}^{k}T^{*}(N^{-})^{(l)}\rightarrow T^{*}\mathrm{GL}_{n}(\mathbb{C}[z]_{k})^{1}$
	by 
	\[
		F((\alpha_{l})_{l=2,3,\ldots,k}):=\omega_{H_{\mathrm{irr}},\theta(\alpha)^{[2;k]}}
		+\sum_{l=2}^{k}L^{*}_{(\theta(\alpha)^{[l+1;k]})^{-1}}\circ R^{*}_{(\theta(\alpha)^{[2;l-1]})^{-1}}\circ T^{*}\pi_{(N^{-})^{(l)}}
		(\alpha_{l})
	\]
	for $(\alpha_{l})_{l=2,3,\ldots,k}\in \prod_{l=2}^{k}T^{*}(N^{-})^{(l)}$, 
	where we put
	\[
		\theta(\alpha)^{[i;j]}:=\begin{cases}
			\theta(\alpha_{j})\theta(\alpha_{j-1})\cdots \theta(\alpha_{i})&i\le j\\
			e&i>j
		\end{cases}.
	\]	
	The next theorem shows that $T^{*}(N^{-})^{(l)}$ defines an affine coordinate system of $\mathbb{O}_{H_{\mathrm{irr}}}$.
	\begin{thm}[Yamakawa \cite{Yam2}, cf. Hiroe-Yamakawa \cite{HY}]\label{thm:stepwise}
		The map $$\nu_{\mathrm{GL}_{n}(\mathbb{C}[z]_{k})^{1}}\circ F\colon \prod_{l=2}^{k}T^{*}(N^{-})^{(l)}
		\rightarrow (\mathfrak{gl}_{n}(\mathbb{C}[z]_{k})^{1})^{*}$$ gives a symplectic isomorphism 
		\[
			\prod_{l=2}^{k}T^{*}(N^{-})^{(l)}\cong \mathbb{O}_{H_{\mathrm{irr}}}.
		\]
	\end{thm}
	Let us give a description of the map $\mathrm{res}_{\mathbb{O}_{H_{\mathrm{irr}}}}\colon \mathbb{O}_{H_{\mathrm{irr}}}
	\rightarrow \mathfrak{gl}_{n}(\mathbb{C})_{H_{\mathrm{irr}}}^{*}$ under the above identification.
	\begin{prp}\label{prop:stepwise}
		Under the isomorphism $\prod_{l=2}^{k}T^{*}(N^{-})^{(l)}\cong \mathbb{O}_{H_{\mathrm{irr}}}$,
		the image of each $(\alpha_{l})_{l=2,3,\ldots,k}\in \prod_{l=2}^{k}T^{*}(N^{-})^{(l)}$
		by the map $\mathrm{res}_{\mathbb{O}_{H_{\mathrm{irr}}}}\colon \mathbb{O}_{H_{\mathrm{irr}}}
		\rightarrow \mathfrak{gl}_{n}(\mathbb{C})_{H_{\mathrm{irr}}}^{*}$
		is given by 
		\begin{multline*}
			\mathrm{res}_{\mathbb{O}_{H_{\mathrm{irr}}}}((\alpha_{l})_{l=2,3,\ldots,k})=\\
			\iota_{0,\,H_{\mathrm{irr}}}^{*}\circ\left(\sum_{l=2}^{k}
			\mathrm{Ad}^{*}(\theta(\alpha)^{[2;l-1]})(T^{*}(\pi_{(N^{-})^{(l)}}\circ \pi_{1})(L^{*}_{\theta(\alpha_{l})}(\alpha_{l})))
			\right).
		\end{multline*}
	\end{prp} 
	\begin{proof}
		Directly follows from the definition of $\mathrm{res}_{\mathbb{O}_{H_{\mathrm{irr}}}}$ and Theorem \ref{thm:stepwise}.
	\end{proof}
	\subsection{Moduli space of meromorphic connections on a trivial bundle over the Riemann sphere 
	with unramified irregular singularities}
	Now we are ready to define a moduli space of meromorphic connections on a trivial bundle over the Riemann sphere
	as a symplectic reduction of a product of truncated orbits.
	
	\subsubsection{Algebraic meromorphic connections on a trivial bundle over the Riemann sphere}
	Let us denote the sheaf of regular function on $\mathbb{P}^{1}$ by $\mathcal{O}^{\mathrm{reg}}_{\mathbb{P}^{1}}$.
	Let $z$ be the standard coordinate on $\mathbb{C}\subset \mathbb{P}^{1}$.
	For 
	$c\in \mathbb{P}^{1}$, set $z_{c}:=z-c$ if $c\in \mathbb{C}$ and $\displaystyle z_{\infty}:=\frac{1}{z}$ if $c=\infty$.
	Let $\widehat{\mathcal{O}}_{c}:=\mathbb{C}[\![z_{c}]\!]$ be the ring of formal power series 
	and $\widehat{\mathcal{M}}_{c}:=\mathbb{C}(\!(z_{c})\!)$ the field of formal Laurent series around $c\in \mathbb{P}^{1}$.
	Also set $\mathcal{O}_{c,i}:=\mathbb{C}[z-c]_{i}$ and $\mathcal{M}_{c,i}:=\mathbb{C}[(z-c)^{-1}]_{i}$
	for a positive integer $i$.
	
	Let us consider a finite set $\{a_{0},a_{1},\ldots,a_{d}\}$
	of points in $\mathbb{P}^{1}$ and fix an effective divisor 
	$D:=\sum_{a\in \{a_{0},a_{1},\ldots,a_{d}\}}(k_{a}+1)\cdot a$
	with $k_{a}\in\mathbb{Z}_{\ge 0}$.
	We denote the set of points by $|D|:=\{a_{0},a_{1},\ldots,a_{d}\}$.
	Let $\mathcal{O}^{\mathrm{reg}}_{\mathbb{P}^{1}}(*|D|)$ and $\varOmega^{\mathrm{reg}}_{\mathbb{P}^{1}}(*|D|)$ denote 
	the sheaves of rational function and of rational 1-forms with poles on $|D|$ respectively.
	
	\begin{dfn}
		An algebraic {\em meromorphic connection} on the trivial bundle is a differential operator,
		namely morphism of sheaves of $\mathbb{C}$-vector spaces
		\[
		\nabla\colon (\mathcal{O}^{\mathrm{reg}}_{\mathbb{P}^{1}})^{n}\longrightarrow (\mathcal{O}^{\mathrm{reg}}_{\mathbb{P}^{1}})^{n}\otimes_{\mathcal{O}^{\mathrm{reg}}_{\mathbb{P}^{1}}}\varOmega^{\mathrm{reg}}_{\mathbb{P}^{1}}(*|D|),
		\]
		satisfying the Leibniz rule
		\[
			\nabla(fs)=df\otimes s+f\nabla(s)	
		\]
		for all open subsets $U\subset \mathbb{P}^{1}$, $f\in \mathcal{O}^{\mathrm{reg}}_{\mathbb{P}^{1}}(U)$, and $s\in (\mathcal{O}^{\mathrm{reg}}_{\mathbb{P}^{1}})^{n}(U)$.
	On any open subset $U\subset \mathbb{P}^{1}$,  $\nabla$ can be uniquely written as the matrix form,
	\[
	\nabla=d-A_{U}(z)dz 
	\]
	where $A_{U}(z)dz\in M_{n}(\varOmega^{\mathrm{reg}}_{\mathbb{P}^{1}}(*|D|)(U))$.
	Thus there uniquely exists the matrix 1-form $A(z)dz\in M_{n}(\varOmega^{\mathrm{reg}}_{\mathbb{P}^{1}}(*|D|)(\mathbb{P}^{1}))$
	defined globally on $\mathbb{P}^{1}$ such that 
	\[
		A(z)dz|_U=A_{U}(z)dz 
	\]
	for any open subset $U\subset \mathbb{P}^{1}$. We call this matrix $A(z)$ the {\em coefficient matrix} of the connection $\nabla$.
	\end{dfn}
	
	The automorphism group of the free $\mathcal{O}^{\mathrm{reg}}_{\mathbb{P}^{1}}$-module $(\mathcal{O}^{\mathrm{reg}}_{\mathbb{P}^{1}})^{n}=
	\mathbb{C}^{n}\otimes_{\mathbb{C}}\mathcal{O}^{\mathrm{reg}}_{\mathbb{P}^{1}}$
	is isomorphic to $\mathrm{GL}_{n}(\mathbb{C})$. Thus
	meromorphic connections $\nabla_{1}=d-A_{1}(z)dz$ and $\nabla_{2}=d-A_{2}(z)dz$ of rank $n$ are isomorphic if and only if there exists $g\in
	\mathrm{GL}_{n}(\mathbb{C})$ such that
	\[
		A_{2}(z)=gA_{1}(z)g^{-1}.
	\]
	A connection $\nabla$ is said to be {\em irreducible} if there is no
	nontrivial subspace $W\subset \mathbb{C}^{n}$ such that the sub $\mathcal{O}^{\mathrm{reg}}_{\mathbb{P}^{1}}$-module $W\otimes_{\mathbb{C}}\mathcal{O}^{\mathrm{reg}}_{\mathbb{P}^{1}}\subset \mathbb{C}^{n}\otimes_{\mathbb{C}}\mathcal{O}^{\mathrm{reg}}_{\mathbb{P}^{1}}=(\mathcal{O}^{\mathrm{reg}}_{\mathbb{P}^{1}})^{n}$ is $\nabla$-stable.
	
	\subsubsection{The space of meromorphic connections as a complex manifold}
	By the projective transformation on $\mathbb{P}^{1}(\mathbb{C})$, we may suppose that $|D|\subset \mathbb{C}$.
	Let us consider a space of connections with coefficients in $\varOmega^{\mathrm{reg}}_{\mathbb{P}^{1}}(*|D|)$, i.e.,
	\[
		\mathcal{M}_{D}:=\{\nabla=d-A(z)dz\mid A(z)dz\in M_{n}(\varOmega^{\mathrm{reg}}_{\mathbb{P}^{1}}(*|D|)(\mathbb{P}^{1}))\}.  
	\]
	For $d-A(z)dz\in \mathcal{M}_{D}$, we can write 
	\[
		A(z)=\sum_{a\in |D|}\mathrm{prc}_{a}(A(z))=\sum_{a\in |D|}\sum_{i=0}^{k_{a}}\frac{A^{(a)}_{i}}{(z-a)^{i}}\frac{dz}{z-a}.
	\]
	with the relation 
	\[
		\sum_{a\in |D|}A^{(a)}_{0}=0.
	\]
	Thus we can identify $\mathcal{M}_{D}$ as a space of matrices,
	\[
		M_{D}:=\left\{\left(X^{(a)}_{i}\right)_{\substack{a\in |D|,\\ i=0,1,\ldots,k_{a}}}\in \prod_{a\in |D|}\prod_{i=0}^{k_{a}}M_{n}(\mathbb{C})\,\middle|\, 
		\sum_{a\in |D|}\sum_{i=0}^{k_{a}}X^{(a)}_{i}=0\right\},
	\]	
	which is a union of hyperplanes in the affine space $\prod_{a\in |D|}\prod_{i=0}^{k_{a}}M_{n}(\mathbb{C})$, i.e., $M_{D}$ is a complex manifold.
	The $\mathrm{GL}_{n}(\mathbb{C})$-action on $\mathcal{M}_{D}$ is translated into the diagonal 
	action on $M_{D}$, i.e.,
	\[
		\mathrm{GL}_{n}(\mathbb{C})\times M_{D}\ni \left(g, \left(X^{(a)}_{i}\right)_{\substack{a\in |D|,\\ i=0,1,\ldots,k_{a}}}\right)
		\longmapsto \left(gX^{(a)}_{i}g^{-1}\right)_{\substack{a\in |D|,\\ i=0,1,\ldots,k_{a}}}\in M_{D}.
	\]
	Under this identification, we regard $\mathcal{M}_{D}$ as a complex manifold with holomorphic $\mathrm{GL}_{n}(\mathbb{C})$-action.
	The irreducibility of connections on $\mathcal{M}_{D}$ corresponds to the following condition on $M_{D}$, namely, $\left(X^{(a)}_{i}\right)_{\substack{a\in |D|,\\ i=0,1,\ldots,k_{a}}}
	\in \prod_{a\in |D|}\prod_{i=0}^{k_{a}}M_{n}(\mathbb{C})$ is said to be {\em irreducible} if it has no nontrivial simultaneous invariant subspace of $\mathbb{C}^{n}$.
	That is to say, if a subspace $W\subset \mathbb{C}^{n}$ satisfies $X^{(a)}_{i}W\subset W$ for all $a\in |D|$ and $i=0,1,\ldots,k_{a}$, then
	$W=\{0\}$ or $W=\mathbb{C}^{n}$.
	Let us denote subspaces of irreducible elements in $\mathcal{M}_{D}$ and $M_{D}$
	by $\mathcal{M}_{D}^{s}$ and $M^{s}_{D}$ respectively.
	
	\subsubsection{Residue maps of truncated orbits}\label{sec:red}
	Let $\mathbb{O}_{H}$ be the truncated $\mathrm{GL}_{n}(\mathbb{C}[z]_{k})$-orbit of 
	an unramified HTL normal form $H
	\in M_{n}(\mathbb{C}[z^{-1}]_{k})$ 
	and recall that the map
	\[
		\mu_{\mathbb{O}_{H}}\colon \mathbb{O}_{H}\cong \mathrm{GL}_{n}(\mathbb{C}[z]_{k})/\mathrm{GL}_{n}(\mathbb{C}[z]_{k})_{H}\ni [g]
		\longmapsto -\mathrm{Ad}^{*}(g)(H)\in M_{n}(\mathbb{C}[z^{-1}]_{k}),
	\]
	is a moment map with respect to the action of $\mathrm{GL}_{n}(\mathbb{C}[z]_{k})$.
	Then the map 
	$\mu_{\mathbb{O}_{H_{a}}}^{0}:=\mathrm{res}_{z=a}\circ \mu_{\mathbb{O}_{H_{a}}}\colon \mathbb{O}_{H_{a}}
	\rightarrow M_{n}(\mathbb{C})\cong \mathfrak{gl}_{n}(\mathbb{C})^{*}$
	becomes a moment map  with respect to the action of $\mathrm{GL}_{n}(\mathbb{C}).$
	Let us consider the map  
	\[
		\nu_{\mathrm{GL}_{n}(\mathbb{C})}\circ \mathrm{pr}_{1}\colon T^{*}\mathrm{GL}_{n}(\mathbb{C})\times \mathbb{O}_{H_{\mathrm{irr}}}\rightarrow 
		\mathfrak{gl}_{n}(\mathbb{C})^{*}
	\]
	and its restriction 
	$\nu_{\mathrm{GL}_{n}(\mathbb{C})}\circ \mathrm{pr}_{1}|_{\mu_{\mathrm{ext}}^{-1}(H_{\mathrm{res}})}$
	on the subspace $\mu_{\mathrm{ext}}^{-1}(H_{\mathrm{res}})\subset T^{*}\mathrm{GL}_{n}(\mathbb{C})\times \mathbb{O}_{H_{\mathrm{irr}}}$.
	Then since $\nu_{\mathrm{GL}_{n}(\mathbb{C})}\circ \mathrm{pr}_{1}|_{\mu_{\mathrm{ext}}^{-1}(H_{\mathrm{res}})}$ 
	is $(\mathrm{GL}_{n}(\mathbb{C})_{H_{\mathrm{irr}}})_{H_{\mathrm{res}}}$-invariant,
	there uniquely exists the map 
	$\nu_{\mathbb{O}_{H}}\colon \mathbb{O}_{H}\rightarrow \mathfrak{gl}_{n}(\mathbb{C})^{*}$ such that 
	the diagram 
	\[
		\begin{tikzcd}
			\mu_{\mathrm{ext}}^{-1}(H_{\mathrm{res}}) \arrow[r,"\nu_{\mathrm{GL}_{n}(\mathbb{C})}\circ \mathrm{pr}_{1}"] \arrow[d]&
			[2 em]\mathfrak{gl}_{n}(\mathbb{C})^{*}\\
			\mathbb{O}_{H}\arrow[ur, "\nu_{\mathbb{O}_{H}}"']&
		\end{tikzcd}
	\]
	is commutative. Then by Theorem \ref{thm:boalch}, we obtain the following.
	\begin{prp}\label{prop:red}
		The map $-\nu_{\mathbb{O}_{H}}\colon \mathbb{O}_{H}\rightarrow \mathfrak{gl}_{n}(\mathbb{C})^{*}$ defined above 
		coincides with the moment map $\mu_{\mathbb{O}_{H}}^{0}\colon \mathbb{O}_{H}\rightarrow \mathfrak{gl}_{n}(\mathbb{C})^{*}$.
	\end{prp}
	
	\subsubsection{Moduli space of meromorphic connections on a trivial bundle with unramified irregular sigularities}
	For each $a\in |D|$, let us take
	an unramified HTL normal form
	\[
	H_{a}:=\left(\frac{S_{k_{a}}^{(a)}}{z_{a}^{k_{a}-1}}+\cdots+\frac{S^{(a)}_{1}}{z_{a}}+S^{(a)}_{0}+N^{(a)}_{0}\right)
	\frac{dz_{a}}{z_{a}},
	\]
	such that 
	$
		\sum_{a\in |D|}\mathrm{res}_{z=a}\mathrm{tr}(H_{a})=0.$	
	We denote the collection of these HTL normal forms by 
	$\mathbf{H}:=(H_{a})_{a\in |D|}$.
	
	Let us consider the map 
	\[
		\begin{array}{cccc}
			\mu_{\mathbf{H}}\colon &\prod_{a\in |D|}\mathbb{O}_{H_{a}}&\longrightarrow &\mathfrak{gl}_{n}(\mathbb{C})^{*}\\
			&(X_{a})_{a\in |D|}&\longmapsto &\sum_{a\in |D|}\mu_{\mathbb{O}_{H_{a}}}^{0}(X_{a})
		\end{array}
	\]
	which is a moment map with respect to the diagonal action of $\mathrm{GL}_{n}(\mathbb{C})$ on the product space $\prod_{a\in |D|}\mathbb{O}_{H_{a}}$.
	
	By the injective immersion
	\[
		\begin{array}{cccc}
		\iota_{\prod_{a\in |D|}{\mathbb{O}_{H_{a}}}}\colon &\prod_{a\in |D|}\mathbb{O}_{H_{a}}&\longrightarrow &\prod_{a\in |D|}M_{n}(\mathbb{C}[z_{a}^{-1}]_{k_{a}})\\
		&([g_{a}])_{a\in |D|}&\longmapsto &(\mathrm{Ad}^{*}(g_{a})(H_{a}))_{a\in |D|}
		\end{array},
	\]
	we regard $\prod_{a\in |D|}\mathbb{O}_{H_{a}}$ as 
	a subspace of $\prod_{a\in |D|}\prod_{i=0}^{k_{a}}M_{n}(\mathbb{C})$ and 
	we denote the space of irreducible elements in $\prod_{a\in |D|}\mathbb{O}_{H_{a}}\subset \prod_{a\in |D|}\prod_{i=0}^{k_{a}}M_{n}(\mathbb{C})$
	by $\left(\prod_{a\in |D|}\mathbb{O}_{H_{a}}\right)^{s}$.
	
	\begin{dfn}
	The symplectic quotient space
		\[
			\mathcal{M}^{*}_{s}(\mathbf{H}):=\mathrm{GL}_{n}(\mathbb{C})\backslash (\mu_{\mathbf{H}}^{s})^{-1}(0)
		\]
		is 
		the {\em moduli space} of meromorphic connections on the rank $n$ trivial bundle on $\mathbb{P}^{1}$
		with respect to the unramified HTL normal forms $\mathbf{H}$.
		It is known that 
		if $(\mu_{\mathbf{H}}^{s})^{-1}(0)\neq \emptyset$, then the moduli space $\mathcal{M}^{*}_{s}(\mathbf{H})$ is a holomorphic symplectic manifold.
	\end{dfn}
	
	\section{Construction of unfolding manifold}
	This section is the main part of the note. 
	Let us consider a moduli space of meromorphic connection $\mathcal{M}^{*}_{s}(\mathbf{H})$
	with respect to a collection of unramified HTL normal forms $\mathbf{H}$.
	First we introduce a holomorphic family of 
	collections of unramified HTL normal forms which describes 
	unfolding procedure of irregular singularities of HTL normal forms in 
	$\mathbf{H}$.
	We call this family the unfolding of $\mathbf{H}$.
	Then we shall explain the construction of 
	a Poisson manifold whose 
	symplectic leaves are isomorphic to Zariski open subsets of the moduli spaces of 
	meromorphic connections with respect to the collections of HTL normal forms 
	appearing in this unfolding family.
	
	\subsection{Deformation of HTL normal forms and unfolding of spectral types}
	
	\subsubsection{An open subset of $\mathbb{C}^{k+1}$ associated to $H$}\label{sec:openset}
	Let us consider an unramified HTL normal form
	\[
		H=\left(\frac{S_{k}}{z^{k}}+\cdots+\frac{S_{1}}{z}+S_{0}+N_{0}\right)\frac{dz}{z}
		\in M_{n}(\mathbb{C}[z^{-1}]_{k})dz
	\]
	with the spectral type 
	\[
		\mathrm{sp}(H)=(\mathbf{m}_{0}\le_{\phi_{0}}\le  \mathbf{m}_{1}\le_{\phi_{1}} \cdots \le_{\phi_{k-1}}\mathbf{m}_{k}, \sigma(\mathbf{m}_{0})).
	\]
	Let $\mathbb{C}^{n}=\bigoplus_{j=1}^{m(i)}V_{\langle i,j\rangle}$
	be the direct sum of simultaneous eigenspaces
	of $(S_{i},S_{i+1},\ldots,S_{k})$ defined in Section \ref{sec:spec}.  
	We denote the eigenvalues of $S_{l}$ on $V_{\langle i,j\rangle}$
	by $s_{\langle i,j\rangle}^{(l)}$ for $i\le l\le k$.
	Let us introduce a distance function $d_{i}(\cdot,\cdot)$
	on $\{1,2,\ldots,m(i)\}$ for each $i=0,1,\ldots,k$ as follows,
	\[
		d_{i}(j,j'):=\mathrm{min}\{l\in\{i,i+1,\ldots,k\}\mid
		\phi^{l}_{i}(j)=\phi^{l}_{i}(j')\}-i,
	\]
	where $\phi_{i}^{l}$ is the composition of refinement maps
	$\phi_{l}\circ\cdots \circ\phi_{i+1}\circ \phi_{i}$
	and we formally set $\phi^{i}_{i}:=\phi_{i}$ and 
	$\phi_{k}\colon \{1,2,\ldots,m(k)\}\rightarrow\{1\}$.
	Then we note that 
	\[
		\begin{cases}
			V_{\langle l,\phi_{i}^{l}(j)\rangle}=V_{\langle l,\phi_{i}^{l}(j')\rangle}& d_{i}(j,j')+i\le l\le k,\\
			V_{\langle l,\phi_{i}^{l}(j)\rangle}\neq V_{\langle l,\phi_{i}^{l}(j')\rangle}&l=d_{i}(j,j')+i-1.
		\end{cases}	
	\]
	Since $V_{\langle i,j\rangle}$ is a direct summand of $V_{\langle l,\phi_{i}^{l}(j)\rangle}$, this implies that 
	eigenvalues $s_{\langle i,j'\rangle}^{(l)}$ satisfy the following relations,
	\begin{equation}\label{eq:eigen}
		\begin{cases}
			s_{\langle i,j'\rangle}^{(l)}=s_{\langle i,j'\rangle}^{(l)}
		&d_{i}(j,j')+i\le l\le k,\\
		s^{(l)}_{\langle i,j\rangle}\neq s^{(l)}_{\langle i,j'\rangle}&l=d_{i}(j,j')+i-1.
		\end{cases}
	\end{equation}
	Let us define polynomials
	\begin{align*}
		\alpha(x_{0},x_{1},\ldots,x_{k})_{\langle i,j\rangle}:=
		\sum_{i\le l\le k}s^{(l)}_{\langle i,j\rangle}\prod_{l<\nu\le k}(x_{i}-x_{\nu}),
	\end{align*}
	and hypersurfaces
	\[
		D_{i;j,j'}:=\left\{
			\mathbf{c}\in \mathbb{C}^{k+1}\, \middle|\, \alpha(\mathbf{c})_{\langle i,j\rangle}
			= \alpha(\mathbf{c})_{\langle i,j'\rangle}
			\right\}
	\]
	for $i=0,1,\ldots,k-1$ and $j\neq j'\in \{1,2,\ldots,m(i)\}$.  
	Then $(\ref{eq:eigen})$ implies that 
	\begin{equation*}
		\begin{split}
			&\alpha(x_{0},x_{1},\ldots,x_{k})_{\langle i,j\rangle}
			-\alpha(x_{0},x_{1},\ldots,x_{k})_{\langle i,j'\rangle}\\
			&\quad =\sum_{i\le l \le i+d_{i}(j,j')-1}(s^{(l)}_{\langle i,j\rangle }-s^{(l)}_{\langle i,j'\rangle})
			\prod_{l<\nu\le k}(x_{i}-x_{\nu})\\
			&\quad =\prod_{d_{i}(j,j')+i\le \mu\le k}(x_{i}-x_{\nu})\left(\sum_{i\le l \le i+d_{i}(j,j')-1}(s^{(l)}_{\langle i,j\rangle }-s^{(l)}_{\langle i,j'\rangle})
			\prod_{l<\nu\le i+ d_{i}(j,j')-1}(x_{i}-x_{\nu})\right).
		\end{split}
	\end{equation*}
	This leads us to define the following polynomials
	\[
		\alpha^{*}(x_{0},x_{1},\ldots,x_{k})_{i;j,j'}:=	\sum_{i\le l \le i+d_{i}(j,j')-1}(s^{(l)}_{\langle i,j\rangle }-s^{(l)}_{\langle i,j'\rangle})
		\prod_{l<\nu\le i+ d_{i}(j,j')-1}(x_{i}-x_{\nu}).
	\]
	Then we define the subset of $D_{i;j,j'}$ by 
	\[
		D_{i;j,j'}^{*}:=\left\{
		\mathbf{c}\in \mathbb{C}^{k+1}\,\middle|\,
		\alpha^{*}(\mathbf{c})_{i;j,j'}=0
		\right\}\subset D_{i;j,j'}.
	\]
	Then we define the Zariski open subset of $\mathbb{C}^{k+1}$ by
	\begin{equation}\label{eq:deformationspace}
		\mathbb{D}(H):=\mathbb{C}^{k+1}-\left(\bigcup_{i=0}^{k-1}\bigcup_{\substack{j,j'\in\{1,2,\ldots,m(i)\} \\j< j'}}D_{i;j,j'}^{*}\right).
	\end{equation}
	We note that $\mathbb{D}(H)$ contains the origin $\mathbf{0}\in 
	\mathbb{C}^{k+1}$.
	
	\subsubsection{A stratification of $\mathbb{C}^{k+1}$ associated to partitions of a finite set}\label{sec:strata}
	Firstly we introduce a stratification of $\mathbb{C}^{k+1}$
	which is induced from partitions of the finite set $\{0,1,2,\ldots,k\}$.
	Let
	$\mathcal{I}\colon I_{1},I_{2},\ldots,I_{r}$  be a partition of $\{0,1,2,\ldots,k\}$,
	i.e., the direct sum decomposition
	$I_{0}\sqcup I_{1}\sqcup\cdots\sqcup I_{r}=\{0,1,2,\ldots,k\}.$
	We may suppose that elements in $I_{j}=\{i_{[j,0]},i_{[j,1]},\ldots,i_{[j,k_{j}]}\}$
	are arranged in the ascending order,
	i.e.,
	\[
		i_{[j,0]}<i_{[j,1]}<\ldots<i_{[j,k_{j}]}
	\]
	as positive integers. Also we assume that  $0\in I_{0}$.
	Along with the partition $\mathcal{I}$, we define an embedding of $\mathbb{C}^{r+1}$ into $\mathbb{C}^{k+1}$,
	\[
		\iota_{\mathcal{I}}\colon \mathbb{C}^{r+1}\ni\mathbf{a}=(a_{0},a_{1},\ldots,a_{r})\mapsto 
		(\iota_{\mathcal{I}}(\mathbf{a})_{0},\iota_{\mathcal{I}}(\mathbf{a})_{1},\ldots,\iota_{\mathcal{I}}(\mathbf{a})_{k})\in \mathbb{C}^{k+1}
	\]
	by setting 
	\[
		\iota_{\mathcal{I}}(\mathbf{a})_{i}:=a_{l}\quad (i\in I_{l}).
	\]
	Let us denote the configuration space of $r+1$ points in $\mathbb{C}$ by 
	\[
		C_{r+1}(\mathbb{C}):=\{(a_{0},a_{1},\ldots,a_{r})\in\mathbb{C}^{r+1}\mid a_{i}\neq a_{j}\text{ for }i\neq j\}.	
	\]
	Then we define a subspace of $\mathbb{C}^{k}$ by 
	\[
		C(\mathcal{I}):=\iota_{\mathcal{I}}(C_{r+1}(\mathbb{C}))=
		\left\{
			(a_{0},a_{1},\ldots,a_{k})\in \mathbb{C}^{k+1}\,\middle|\,
			\begin{array}{ll}
				a_{i}=a_{j}&\text{if }i,j\in I_{l}\text{ for some }l,\\
				a_{i}\neq a_{j}&\text{otherwise}
			\end{array}
		\right\}.
	\]
	Then we obtain the direct sum decomposition
	\begin{equation}\label{eq:stra}
			\mathbb{C}^{k+1}=\bigsqcup_{\mathcal{I}\in \mathcal{P}_{[k+1]}}C(\mathcal{I}),
	\end{equation}
	where $\mathcal{P}_{[k+1]}$ is the set of all partitions of $\{0,1,\ldots,k\}$.
	The set $\mathcal{P}_{[k+1]}$ of partitions is 
	naturally equipped with the partial order defined by the 
	refinement of the partitions and the each direct summand satisfies the closure relation
	\[
		\overline{C(\mathcal{I})}=\bigsqcup_{\substack{\mathcal{I}'\in \mathcal{P}_{[k+1]},\\\mathcal{I}'\le \mathcal{I}}}C(\mathcal{I}').
	\]
	Thus the decomposition $(\ref{eq:stra})$ gives a stratification of $\mathbb{C}^{k+1}$ associated to the poset $\mathcal{P}_{[k+1]}$.
	
	\subsubsection{Unfolding of an HTL normal form}\label{sec:unfoldspec}
	
	\begin{dfn}[Unfolding of an HTL normal form]\label{def:unfold}
	The  {\em unfolding of the irregular type} $H_{\mathrm{irr}}$ of $H$ is the function from $\mathbb{C}^{k+1}$
	to $M_{n}(\mathbb{C}(z))dz$
	of the form,
	\begin{multline*}
		H_{\mathrm{irr}}(c_{0},c_{1},\ldots,c_{k}):=\\
		\left(\frac{S_{k}}{(z-c_{1})\cdots(z-c_{k})}
		+\frac{S_{k-1}}{(z-c_{1})\cdots(z-c_{k-1})}+\cdots +\frac{S_{1}}{z-c_{1}}
		\right)\frac{dz}{z-c_{0}}.
	\end{multline*}
	We also define the {\em unfolding of $H$} by 
	\[
		H(c_{0},c_{1},\ldots,c_{k}):=H_{\mathrm{irr}}(c_{0},c_{1},\ldots,c_{k})+(S_{0}+N_{0})\frac{dz}{z-c_{0}}	
	\]
	and call $H_{\mathrm{irr}}(c_{0},c_{1},\ldots,c_{k})$ the {\em irregular part of the unfolding} $H(c_{0},c_{1},\ldots,c_{k})$.
	\end{dfn}
	Let $\mathbb{C}^{k+1}=\bigsqcup_{\mathcal{I}\in \mathcal{P}_{[k+1]}}C(\mathcal{I})$ be the stratification of $\mathbb{C}^{k+1}$
	defined in the previous section. We shall compute the spectral types of $H(c_{1},c_{2},\ldots,c_{k})$
	for the parameters $(c_{0},c_{1},\ldots,c_{k})$ on these strata $C(\mathcal{I})$.
	
	Fix a partition
	$\mathcal{I}\colon I_{0},I_{1},\ldots,I_{r} \in \mathcal{P}_{[k+1]}$ and
	consider the embedding $\iota_{\mathcal{I}}\colon \mathbb{C}^{r+1}\hookrightarrow \mathbb{C}^{k+1}$
	defined previously.
	Then since $C(\mathcal{I})=\{\iota_{\mathcal{I}}(\mathsf{c})\in \mathbb{C}^{k+1}\mid \mathsf{c}\in C_{r+1}(\mathbb{C})\}$,
	we can introduce a parameterization of the unfolding of $H$ on $C(\mathcal{I})$ by 
	setting
	\[
		H_{\mathcal{I}}(\mathsf{c}):=H(\iota_{\mathcal{I}}(\mathsf{c}))\quad (\mathsf{c}\in C_{r+1}(\mathbb{C}))
	\]
	Let us write $(c_{0},c_{1},\ldots,c_{r})\in C_{r}(\mathbb{C})$. Then $H_{\mathcal{I}}(\mathsf{c})\in M_{n}(\mathbb{C}(z))dz$
	has poles at $c_{j}$ of order at most $k_{j}+1=|I_{j}|$ for $j=0,1,\ldots,r$. Thus 
	the partial fractional decomposition algorithm gives the description  
	$H_{\mathcal{I}}(\mathsf{c})=\sum_{j=0}^{r}
		\sum_{\nu=0}^{k_{i}}\frac{A^{[j]}_{\nu}}{z_{c_{j}}^{\nu+1}}
		\,dz_{c_{j}}$
	where $A^{[j]}_{\nu}\in M_{n}(\mathbb{C})$ and $z_{c_{j}}=z-c_{j}$. We denote each component of the sum by 
	\[
		H_{\mathcal{I}}(\mathsf{c})_{z_{c_{j}}}:=\sum_{\nu=0}^{k_{i}}\frac{A^{[j]}_{\nu}}{z_{c_{j}}^{\nu+1}}
	\,dz_{c_{j}}.
	\]
	
	\begin{prp}[\cite{H}]\label{prop:spectral type}
		Let us suppose that $\iota_{\mathcal{I}}(\mathsf{c})\in \mathbb{D}(H)$.
		Then $H_{\mathcal{I}}(\mathsf{c})_{z_{c_{j}}}$
		become unramified HTL-normal forms with 
		the spectral types 
	\[
		\begin{cases}
		(\mathbf{m}_{i_{[j,0]}}\le_{\phi_{i_{[j,0]}}^{i_{[j,1]}}}\mathbf{m}_{i_{[j,1]}}\le_{\phi_{i_{[j,1]}}^{i_{[j,2]}}}\cdots\le_{\phi_{i_{[j,k_{j}-1]}}^{i_{[j,k_{j}]}}} \mathbf{m}_{i_{[j,k_{j}]}},\mathrm{triv})&\text{ for }
	j=1,2,\ldots,r,\\
		(\mathbf{m}_{i_{[0,0]}}\le_{\phi_{i_{[0,0]}}^{i_{[0,1]}}}\mathbf{m}_{i_{[0,1]}}\le_{\phi_{i_{[0,1]}}^{i_{[0,2]}}}\cdots\le_{\phi_{i_{[0,k_{0}-1]}}^{i_{[0,k_{0}]}}} \mathbf{m}_{i_{[0,k_{0}]}},\sigma(\mathbf{m}_{0}))&\text{ for }
	j=0
		\end{cases}.
	\]
	\end{prp}
	This proposition says that the each fiber of the unfolding of $H$
	is a sum of unramified HTL normal forms $H_{\mathcal{I}}(\mathsf{c})_{z_{c_{j}}}$
	and  
	spectral types of them 
	do not depend on the parameter $\mathsf{c}$ but only on the stratum $C(\mathcal{I})
	\subset \mathbb{C}^{k+1}$. 
	
	\subsection{Deformation of $\mathbb{O}_{H_{\mathrm{irr}}}$}
	In the previous section, we introduced the holomorphic family 
	$(H(\mathbf{c}))_{\mathbf{c}\in \mathbb{D}(H)}$
	and explained that each fiber $H(\mathbf{c})$ is a sum of unramified HTL normal forms.
	Next we shall construct a deformation of the truncated orbit $\mathbb{O}_{H}$ 
	in accordance with this family $(H(\mathbf{c}))_{\mathbf{c}\in \mathbb{D}(H)}$.
	For this purpose, we first explain the construction of a deformation of the orbit 
	$\mathbb{O}_{H_{\mathrm{irr}}}$ of the irregular type.
	
	\subsubsection{Deformations of $\mathbb{C}[z]_{l}$ and $\mathbb{C}[z^{-1}]_{l}dz$}
	To a point $\mathbf{c}=(c_{0},c_{1},\ldots,c_{k})\in \mathbb{C}^{k+1}$, we associate
	an effective divisor 
	\[
		D(\mathbf{c}):=c_{0}+c_{1}+\cdots +c_{k}	
	\]
	of $\mathbb{A}^{1}(\mathbb{C})=\mathbb{C}$. Then we consider a subspace of rational 1-forms
	\begin{equation*}
		\varOmega^{\mathrm{rat}}(D(\mathbf{c})):=\left\{f\in \mathbb{C}(z)dz\,\middle|\,
		\mathrm{ord}_{P}(f)\ge -n_{P}\text{ for all }P\in \mathbb{C}\right\}\\
	\end{equation*}
	where $n_{P}$ is the coefficient of $P\in \mathbb{C}$ in 
	the formal sum $D(\mathbf{c})=\sum_{P\in\mathbb{C}}n_{P}P$, and 
	$\mathrm{ord}_{P}(f)$ is the order of $f$ at $P$.
	We can regard 
	$\varOmega^{\mathrm{rat}}(D(\mathbf{c}))$ as $\mathbb{C}[z]$-module containing 
	$\mathbb{C}[z]dz$ as a submodule and consider
	a quotient module 
	\[
		\widehat{\varOmega}(\mathbf{c}):=\varOmega^{\mathrm{rat}}_{\mathbb{C}}(D(\mathbf{c}))/\mathbb{C}[z]dz.	
	\]
	We also define the linear system
	\[
		L^{\mathrm{rat}}(-D(\mathbf{c})):=\left\{f\in \mathbb{C}(z)\mid \mathrm{ord}_{P}(f)
		\ge n_{P}\text{ for all }P\in \mathbb{C}\right\}.
	\]
	Since $D(\mathbf{c})$ is an effective divisor, i.e., $n_{P}\ge 0$ for all $P\in \mathbb{C}$, we may
	regard $L^{\mathrm{rat}}(-D(\mathbf{c}))$ as an ideal of $\mathbb{C}[z]$, namely,
	\[
		L^{\mathrm{rat}}(-D(\mathbf{c}))=\left\langle\prod_{i=0}^{k}(z-c_{i})\right\rangle_{\mathbb{C}[z]},	
	\]
	and take the
	quotient ring 
	\[
		\widehat{\mathbb{C}[z]}(\mathbf{c}):=\mathbb{C}[z]/L^{\mathrm{rat}}(-D(\mathbf{c})).	
	\]
	We denote the class of $f(z)\in \mathbb{C}[z]$
	in $\widehat{\mathbb{C}[z]}(\mathbf{c})$ and that of $g(z)dz\in \varOmega^{\mathrm{rat}}_{\mathbb{C}}(D(\mathbf{c}))$
	in $\widehat{\varOmega}_{\mathbb{C}}(\mathbf{c})$
	by the same notations, if it may not cause any confusion.
	\begin{dfn}[Residue and evaluation maps]
		For $l\in\{0,1,\ldots,k\}$,
		we define the {\em residue map} at $c_{l}$ by
	\[
		\mathrm{res}_{z=c_{l}}([f(z)dz]):=\mathrm{res}_{z=c_{l}}f(z)dz	
	\]
	for $[f(z)dz]\in \widehat{\varOmega}(\mathbf{c})$, the class of $f(z)\in \varOmega_{\mathbb{C}}^{\mathrm{rat}}(D(\mathbf{c}))$.
	Further, we define the residue at $\infty$ by 
	\[
		\mathrm{res}_{z=\infty}([f(z)dz]):=-\mathrm{res}_{w=0}f(1/w)\frac{dw}{w^{2}}.
	\]
	We note the all these residue maps are well-defined in $\widehat{\varOmega}(\mathbf{c})$.
	
	Also we define the {\em evaluation map} $\mathrm{ev}_{c_{j}}\colon \widehat{\mathbb{C}[z]}(\mathbf{c})
		\rightarrow \mathbb{C}$ at $c_{j}$ by 
		\[
			\mathrm{ev}_{c_{j}}\colon \widehat{\mathbb{C}[z]}(\mathbf{c})=\mathbb{C}[z]/\langle \prod_{\mu=0}^{k}(z-c_{\mu})\rangle_{\mathbb{C}[z]}
			\longrightarrow \mathbb{C}[z]/\langle z-c_{j}\rangle_{\mathbb{C}[z]}= \mathbb{C},
		\]
		the natural projection induced by the inclusion $\langle \prod_{\mu=0}^{k}(z-c_{\mu})\rangle_{\mathbb{C}[z]}
		\subset \langle z-c_{j}\rangle_{\mathbb{C}[z]}$.
		Here we identify $\mathbb{C}[z]/\langle z-c_{j}\rangle_{\mathbb{C}[z]}= \mathbb{C}$
		by the unique $\mathbb{C}$-algebra isomorphism.
		Usually we simply write $f(c_{j}):=\mathrm{ev}_{c_{j}}(f(z))$ for $f(z)\in \widehat{\mathbb{C}[z]}(\mathbf{c})$.
	\end{dfn}
	
	Note that $\widehat{\varOmega}(\mathbf{c})$ is 
	a $\widehat{\mathbb{C}[z]}(\mathbf{c})$-module, since elements of $L^{\mathrm{rat}}(-D(\mathbf{c}))
	\subset\mathbb{C}[z]$ act trivially on $\widehat{\varOmega}_{\mathbb{C}}(\mathbf{c})$.
	Since
	we have the isomorphisms
	\[
		\widehat{\mathbb{C}[z]}(\mathbf{0})\cong\mathbb{C}[z]_{k},\quad\quad
		\widehat{\varOmega}(\mathbf{0})\cong \mathbb{C}[z^{-1}]_{k}dz,
	\]
	we can regard $\widehat{\mathbb{C}[z]}(\mathbf{c})$ and $\widehat{\Omega}(\mathbf{c})$
	as deformations of $\mathbb{C}[z]_{k}$ and $\mathbb{C}[z^{-1}]_{k}dz$.
	
	Let us define a filtration on  $\widehat{\varOmega}(\mathbf{c})$
	and cofiltration on $\widehat{\mathbb{C}[z]}(\mathbf{c})$.
	Consider projection maps $\mathrm{pr}^{(l)}\colon \mathbb{C}^{k+1}\ni(x_{0},x_{1},\ldots,x_{k})\mapsto 
	(x_{0},x_{1},\ldots,x_{l})\in \mathbb{C}^{l+1}$
	and define divisors 
	$D(\mathbf{c})_{\le l}:=D(\mathrm{pr}^{(l)}(\mathbf{c}))$  for $l=0,1,\ldots,k$.
	Then we obtain the following filtration
	\[
		\varOmega^{\mathrm{rat}}(D(\mathbf{c})_{\le 0})\subset \varOmega^{\mathrm{rat}}(D(\mathbf{c})_{\le 1})\subset 
		\cdots \subset \varOmega^{\mathrm{rat}}(D(\mathbf{c})_{\le k})=\varOmega^{\mathrm{rat}}(D(\mathbf{c})).	
	\]
	\begin{dfn}[Standard filtration and basis of $\widehat{\varOmega}(\mathbf{c})$]
		Let us set 
		\[
			\widehat{\varOmega}(\mathbf{c})_{l}:=\varOmega^{\mathrm{rat}}(D(\mathbf{c})_{\le l})/
			\mathbb{C}[z]dz\text{ for } l=0,1,\ldots,k.	
		\] The 
		filtration 
		\[
			\widehat{\varOmega}(\mathbf{c})_{0}\subset \widehat{\varOmega}(\mathbf{c})_{1}
			\subset \cdots \subset \widehat{\varOmega}(\mathbf{c})_{k}=\widehat{\varOmega}(\mathbf{c})
		\]
		induced by the above filtration of $\varOmega^{\mathrm{rat}}_{\mathbb{C}}(D(\mathbf{c}))$
		is called the {\em standard filtration} of $\widehat{\varOmega}(\mathbf{c})$.
		The basis of $\widehat{\varOmega}(\mathbf{c})$,
	\[
		\frac{1}{z-c_{0}},\,\frac{1}{(z-c_{0})(z-c_{1})},\ldots,\frac{1}{(z-c_{0})(z-c_{1})\cdots(z-c_{k})}	
	\]
	as $\mathbb{C}$-vector space is called the  {\em standard basis} of $\widehat{\varOmega}(\mathbf{c})$.
	\end{dfn} 
	Let us note that 
	$\frac{1}{z-c_{0}},\,\frac{1}{(z-c_{0})(z-c_{1})},\ldots,\frac{1}{(z-c_{0})(z-c_{1})\cdots(z-c_{l})}$
	becomes a basis of the $l$-th component $\widehat{\varOmega}(\mathbf{c})_{\le l}$
	of the standard filtration for each $l=0,1,\ldots,k$.
	
	Similarly, we can consider the filtration 
	\[
		\langle z-c_{0}\rangle_{\mathbb{C}[z]}\supset \langle (z-c_{0})(z-c_{1})\rangle_{\mathbb{C}[z]} \supset \cdots
	\supset \left\langle\prod_{i=0}^{k}(z-c_{i})\right\rangle_{\mathbb{C}[z]}  	
	\]
	of ideals on $\mathbb{C}[z]$.
	\begin{dfn}[Standard cofiltration and basis of $\widehat{\mathbb{C}[z]}(\mathbf{c})$]
		Let us set 
		\[
			\widehat{\mathbb{C}[z]}(\mathbf{c})_{l}:=\mathbb{C}[z]/\langle\prod_{i=0}^{l}(z-c_{i})\rangle_{\mathbb{C}[z]}
			\text{ for }l=0,1,\ldots,k.
		\]
		Then the sequence of projection maps 
		\[
			\widehat{\mathbb{C}[z]}(\mathbf{c})_{0}\rightarrow \widehat{\mathbb{C}[z]}(\mathbf{c})_{1}
			\rightarrow\cdots\rightarrow \widehat{\mathbb{C}[z]}(\mathbf{c})_{k}=\widehat{\mathbb{C}[z]}(\mathbf{c})
		\]
		induced by the above filtration of ideals of $\mathbb{C}[z]$ is called the 
		{\em standard cofiltration} of $\widehat{\mathbb{C}[z]}(\mathbf{c})$.
		The basis of $\widehat{\mathbb{C}[z]}(\mathbf{c})$,
		\[
		1, (z-c_{0}), (z-c_{0})(z-c_{1}),\ldots, (z-c_{0})(z-c_{1})\cdots (z-c_{k-1})
		\]
		as $\mathbb{C}$-vector space is called the {\em standard basis} of $\widehat{\mathbb{C}[z]}(\mathbf{c})$.
	\end{dfn}
	
	Next we introduce the pairing of $\widehat{\mathbb{C}[z]}(\mathbf{c})_{l}$ and $\widehat{\varOmega}(\mathbf{c})_{l}$
	as $\mathbb{C}$-vector spaces defined by 
	\[
		\begin{array}{cccc}
		\langle\ ,\ \rangle_{\mathbf{c},l}\colon 	&\widehat{\mathbb{C}[z]}(\mathbf{c})_{l} \times \widehat{\varOmega}(\mathbf{c})_{l}&
		\longrightarrow &\mathbb{C}\\
		&(f(z),g(z)dz)&\longmapsto &-\mathrm{res}_{z=\infty}(f(z)\cdot g(z)dz)
		\end{array}
	\]
	for each $l=0,1,\ldots,k$. Then we can show that 
		the paring $\langle\ ,\ \rangle_{\mathbf{c},l}$ is non-degenerate
		and 
		the bases $$1, (z-c_{0}), (z-c_{0})(z-c_{1}),\ldots, (z-c_{0})(z-c_{1})\cdots (z-c_{l-1})$$
		and $$\frac{1}{z-c_{0}},\,\frac{1}{(z-c_{0})(z-c_{1})},\ldots,\frac{1}{(z-c_{0})(z-c_{1})\cdots(z-c_{l})}$$
		are dual bases with respect to this paring.
	
	\subsubsection{Partial fraction decomposition}\label{sec:pfd}
	In Section \ref{sec:unfoldspec}, we saw that each fiber $H_{\mathcal{I}}(\mathsf{c})$
	of the unfolding of $H$
	on a stratum $C(\mathcal{I})\subset \mathbb{C}^{k+1}$
	 decomposes into the sum of unramified HTL normal form by 
	the partial fraction decomposition algorithm. Now we explain that
	on each stratum $C(\mathcal{I})$,
	the partial fraction decomposition gives an decomposition of $\widehat{\mathbb{C}[z]}(\mathbf{c})_{l}$-module 
	$\widehat{\varOmega}(\mathbf{c})_{l}$. 
	
	Let us take a partition $\mathcal{I}\colon I_{1},I_{2},\ldots,I_{r}$
	of the finite set $\{0,1,\ldots,k\}$. 
	Then for each $l=0,1,\ldots,k$, we define a partition
	 $\mathcal{I}^{(l)}\colon I_{1}^{(l)},I_{2}^{(l)},\ldots,I_{r}^{(l)}$ of the subset $\{0,1,\ldots,l\}\subset
	\{0,1,\ldots,k\}$ by 
	$I^{(l)}_{j}:=I_{j}\cap \{0,1,\ldots,l\}$,
	allowing $I^{(l)}_{j}=\emptyset$.
	Let us set 
	$
		k^{(l)}_{j}:=|I^{(l)}_{j}|-1.	
	$
	Let us take $\mathsf{c}=(c_{0},c_{1},\ldots,c_{r})\in C_{r+1}(\mathbb{C})$
	so that  $\iota_{\mathcal{I}}(\mathsf{c})\in C(\mathcal{I})$ .
	Then we have $D(\iota_{\mathcal{I}}(\mathsf{c}))_{\le l}=\sum_{j=0}^{r}(k_{j}^{(l)}+1)\cdot c_{j}$.
	The algorithm of the partial fraction decomposition of rational functions gives us the direct sum decomposition
	$$\widehat{\varOmega}(\iota_{\mathcal{I}}(\mathsf{c}))_{l}=
	\bigoplus_{j=0}^{r}\mathbb{C}[(z-c_{j})^{-1}]_{k_{j}^{(i)}}$$
	as $\widehat{\mathbb{C}[z]}(\iota_{\mathcal{I}}(\mathsf{c}))_{l}$-modules.
	
	We also have a similar decomposition of $\widehat{\mathbb{C}[z]}(\iota_{\mathcal{I}}(\mathsf{c}))_{l}$.
	Since we have $c_{i}\neq c_{j}$ for $i\neq j$, it follows that 
	\begin{align*}
		&\langle (z-c_{i})^{k_{i}^{(l)}} \rangle_{\mathbb{C}[z]}+\langle (z-c_{i'})^{k_{i'}^{(l)}} \rangle_{\mathbb{C}[z]}=\mathbb{C}[z],\quad i\neq i',\\
		&\bigcap_{i=0}^{l}\langle (z-c_{i})^{k_{i}^{(l)}} \rangle_{\mathbb{C}[z]}=L^{\mathrm{rat}}(-D(\iota_{\mathcal{I}}(\mathbf{c}))_{\le l}).
	\end{align*}
	Therefore the Chinese remainder theorem implies that the map
	\[
		\begin{array}{cccc}
		\prod_{i=0}^{l}\mathrm{pr}_{c_{i}}\colon&\widehat{\mathbb{C}[z]}(\iota_{\mathcal{I}}(\mathsf{c}))_{l}&\longrightarrow &\prod_{i=0}^{l}\mathbb{C}[z-c_{i}]_{k_{i}^{(l)}}\\
		&f(z)& \longmapsto & (\mathrm{pr}_{c_{i}}(f(z))_{i=0,1,\ldots,l}
		\end{array}	
	\]
	is an algebra isomorphism.
	Here 
	\[
		\mathrm{pr}_{c_{i}}\colon \widehat{\mathbb{C}[z]}(\iota_{\mathcal{I}}(\mathsf{c}))_{l}\longrightarrow \\
	\mathbb{C}[z]/\langle(z-c_{i})^{k_{i}^{(l)}+1}\rangle_{\mathbb{C}[z]}=\mathbb{C}[z-c_{i}]_{k_{i}^{(l)}}
	\]
	are the natural projections for $i=0,1,\ldots,l$.
	
	If we note that the ideal $\langle(z-c_{i})^{k_{i}^{(l)}+1}\rangle_{\mathbb{C}[z]}$
	annihilates $\mathbb{C}[(z-c_{i})^{-1}]_{k_{i}^{(l)}}$, it follows that the diagrams
	\[
			\begin{tikzcd}
				\widehat{\mathbb{C}[z]}(\iota_{\mathcal{I}}(\mathsf{c}))_{l}\times \mathbb{C}[(z-c_{i})^{-1}]_{k_{i}^{(l)}} \arrow[d,"\mathrm{pr}_{c_{i}}\times \mathrm{id}"]\arrow[r,"\cdot "] &\mathbb{C}[(z-c_{i})^{-1}]_{k_{i}^{(l)}} \arrow[d,"\mathrm{id}"]\\
				\mathbb{C}[z-c_{i}]_{k_{i}^{(l)}}\times\mathbb{C}[(z-c_{i})^{-1}]_{k_{i}^{(l)}}  \arrow[r,"\cdot "]&\mathbb{C}[(z-c_{i})^{-1}]_{k_{i}^{(l)}}
			\end{tikzcd}	
	\]
	are commutative. Here horizontal maps are scalar multiplications on $\mathbb{C}[z-c_{i}]_{k_{i}^{(l)}}$ as $\widehat{\mathbb{C}[z]}(\iota_{\mathcal{I}}(\mathsf{c}))_{l}$
	and $\mathbb{C}[z-c_{i}]_{k_{i}^{(l)}}$ modules respectively.
	Further, we obtain the following commutative diagram,
	\[
		\begin{tikzcd}
			\widehat{\mathbb{C}[z]}(\iota_{\mathcal{I}}(\mathsf{c}))_{l}\times \widehat{\varOmega}(\iota_{\mathcal{I}}(\mathsf{c}))_{l} \arrow[d,"\prod_{j=0}^{r}(\mathrm{pr}_{c_{i}}\times \mathrm{pr}_{c_{i}})"]\arrow[r,"\cdot "] &\widehat{\varOmega}(\iota_{\mathcal{I}}(\mathsf{c}))_{l}\arrow[d,"="]\\
			\prod_{i=0}^{r}\left(\mathbb{C}[z-c_{i}]_{k_{i}^{(l)}}\times\mathbb{C}[(z-c_{i})^{-1}]_{k_{i}^{(l)}} \right) \arrow[r,"\bigoplus_{i=0}^{r}\cdot "]&\bigoplus_{i=0}^{r}\mathbb{C}[(z-c_{i})^{-1}]_{k_{i}^{(l)}}
		\end{tikzcd}.
	\]

	\subsubsection{Lie groupoid as a deformation of $N_{\mathbf{m}_{l}^{l+1}}^{-}(\mathbb{C}[z]_{l-1})^{1}$}
	For a $\mathbb{C}$-subalgebra $\mathfrak{k}\subset M_{n}(\mathbb{C})$
	and $\mathbf{c}=(c_{0},c_{1},\ldots,c_{k+1})\in \mathbb{D}(H)$, we define 
	subalgebras of $M_{n}(\widehat{\mathbb{C}[z]}(\mathbf{c})_{l-1}):=M_{n}(\mathbb{C})\otimes_{\mathbb{C}}
	\widehat{\mathbb{C}[z]}(\mathbf{c})_{l-1}$ by 
	\begin{align*}
		\mathfrak{k}(\widehat{\mathbb{C}[z]}(\mathbf{c})_{l-1})&:=
		\left\{X=X_{0}+\sum_{i=1}^{l-1}X_{i}(z-c_{0})(z-c_{1})\cdots(z-c_{i-1})\,\middle|\, X_{i}\in \mathfrak{k}\right\},\\
		\mathfrak{k}(\widehat{\mathbb{C}[z]}(\mathbf{c})_{l-1})^{1}&:=
		\left\{X\in \mathfrak{k}(\widehat{\mathbb{C}[z]}(\mathbf{c}))_{l-1}\mid X_{0}=0\right\}.
	\end{align*}
	
	Then we have the natural direct sum decomposition as $\mathbb{C}$-algebras,
	\[
		\mathfrak{k}(\widehat{\mathbb{C}[z]}(\mathbf{c})_{l-1})=\mathfrak{k}\oplus \mathfrak{k}(\widehat{\mathbb{C}[z]}(\mathbf{c})_{l-1})^{1}.
	\]
	Moreover  we define $\mathbb{C}$-vector subspaces of $M_{n}(\widehat{\varOmega}(\mathbf{c})_{l-1}):=M_{n}(\mathbb{C})\otimes_{\mathbb{C}}
	\widehat{\varOmega}(\mathbf{c})_{l-1}$ by
	\begin{align*}
		\mathfrak{k}(\widehat{\varOmega}(\mathbf{c})_{l-1})&:=
		\left\{X=\sum_{i=0}^{l-1}\frac{X_{i}}{(z-c_{0})(z-c_{1})\cdots(z-c_{i})}dz\,\middle|\, X_{i}\in \mathfrak{k}\right\},\\
		\mathfrak{k}(\widehat{\varOmega}(\mathbf{c})_{l-1})^{1}&:=
		\{X\in \mathfrak{k}(\widehat{\varOmega}(\mathbf{c})_{l-1})\mid X_{0}=0\}.
	\end{align*}

	Especially for $\mathfrak{n}_{\mathbf{m}_{l}^{l+1}}^{-}\subset M_{n}(\mathbb{C})$, we
	define 
	\begin{align*}
		N_{\mathbf{m}_{l}^{l+1}}^{-}(\widehat{\mathbb{C}[z]}(\mathbf{c})_{l-1})&:=
		\left\{E_{n}+X\mid X\in \mathfrak{n}_{\mathbf{m}_{l}^{l+1}}^{-}(\widehat{\mathbb{C}[z]}(\mathbf{c})_{l-1})\right\},\\
		N_{\mathbf{m}_{l}^{l+1}}^{-}(\widehat{\mathbb{C}[z]}(\mathbf{c})_{l-1})^{1}&:=
		\left\{E_{n}+X\mid X\in \mathfrak{n}_{\mathbf{m}_{l}^{l+1}}^{-}(\widehat{\mathbb{C}[z]}(\mathbf{c})_{l-1})^{1}\right\},
	\end{align*}
	which
	are equipped with the Lie group structures and  
	 their corresponding Lie algebras are isomorphic to $\mathfrak{n}_{\mathbf{m}_{l}^{l+1}}^{-}(\widehat{\mathbb{C}[z]}(\mathbf{c})_{l-1})$
	and $\mathfrak{n}_{\mathbf{m}_{l}^{l+1}}^{-}(\widehat{\mathbb{C}[z]}(\mathbf{c})_{l-1})^{1}$ respectively.
	
	The paring 
	\begin{equation}\label{eq:paring}
		\begin{array}{ccc}
		\mathfrak{n}_{\mathbf{m}_{l}^{l+1}}^{-}(\widehat{\mathbb{C}[z]}(\mathbf{c}))_{l-1}\times 
		\mathfrak{n}_{\mathbf{m}_{l}^{l+1}}^{+}(\widehat{\varOmega}(\mathbf{c}))_{l-1}&
		\longrightarrow& \mathbb{C}\\
		(X,Y)&\longmapsto& -\mathrm{tr}(\mathrm{res}_{z=\infty}(X\cdot Y)).
		\end{array}
	\end{equation}
	is shown to be
	non-degenerate and thus gives the identifications 
	\[
		\mathfrak{n}_{\mathbf{m}_{l}^{l+1}}^{+}(\widehat{\varOmega}(\mathbf{c})_{l-1})\cong 
		\mathfrak{n}_{\mathbf{m}_{l}^{l+1}}^{-}(\widehat{\mathbb{C}[z]}(\mathbf{c})_{l-1})^{*},\ 
		\mathfrak{n}_{\mathbf{m}_{l}^{l+1}}^{+}(\widehat{\varOmega}(\mathbf{c})_{l-1})^{1}\cong 
		(\mathfrak{n}_{\mathbf{m}_{l}^{l+1}}^{-}(\widehat{\mathbb{C}[z]}(\mathbf{c})_{l-1})^{1})^{*}.
	\]
	Similarly, by setting $\mathfrak{p}_{\mathbf{m}_{l}^{l+1}}^{+}=\mathfrak{l}_{\mathbf{m}_{l}}\oplus \mathfrak{n}_{\mathbf{m}_{l}^{l+1}}^{+}$
	and $\mathfrak{p}_{\mathbf{m}_{l}^{l+1}}^{-}=\mathfrak{n}_{\mathbf{m}_{l}^{l+1}}^{-}\oplus \mathfrak{l}_{\mathbf{m}_{l}}$,
	we have
	\[
		\mathfrak{p}^{+}_{\mathbf{m}_{l}^{l+1}}(\widehat{\varOmega}(\mathbf{c})_{l-1})\cong 
		(\mathfrak{p}_{\mathbf{m}_{l}^{l+1}}^{-}(\widehat{\mathbb{C}[z]}(\mathbf{c}))_{l-1})^{*}.
	\]
	
	Now let us consider the family $(N_{\mathbf{m}_{l}^{l+1}}^{-}(\widehat{\mathbb{C}[z]}(\mathbf{c})_{l-1})^{1})
	_{\mathbf{c}\in \mathbb{D}(H)}$ of complex Lie groups
	 and equip this family with a structure of complex Lie groupoid.
	\begin{dfn}[Complex Lie groupoid]
		Let us consider a complex manifold $\Gamma$ and its submanifold $\Gamma_{0}$
		with the inclusion map $i\colon \Gamma_{0}\hookrightarrow \Gamma$. 
		Let us also consider holomorphic surjective submersions $s,t\colon \Gamma\rightarrow \Gamma_{0}$
		satisfying $t\circ i=s\circ i=\mathrm{id}_{\Gamma_{0}}$.
		Further, we consider a holomorphic map 
		$
		m\colon \Gamma^{(2)}:=\{(\gamma_{1},\gamma_{2})\in \Gamma\times \Gamma\mid s(\gamma_{1})=t(\gamma_{2})\}
		\rightarrow \Gamma	
		$ satisfying $s\circ m(g_{1},g_{2})=s(g_{2})$ and $t\circ m(g_{1},g_{2})=t(g_{1})$.
		Then if the following conditions are satisfied, the tuple 
		$(\Gamma,\Gamma_{0},s,t,m)$ is called a {\em complex Lie groupoid}.
		\begin{enumerate}
			\item \textbf{Associativity:} $m(m(g_{1},g_{2}),g_{3})=m(g_{1},m(g_{2},g_{3}))$ for all $(g_{1},g_{2},g_{3})$
			satisfying $s(g_{1})=t(g_{2})$ and $s(g_{2})=t(g_{3})$.
			\item \textbf{Units:} $m(i\circ t(g),g)=g=m(g,i\circ s(g))$ for all $g\in \Gamma$.
			\item \textbf{Inverses:} For all $g\in \Gamma$, there exists $h\in \Gamma$ such that 
			$s(h)=t(g)$, $t(h)=s(g)$ and $m(g,h), m(h,g)\in \Gamma_{0}$.
		\end{enumerate}
		The maps $s$ and $t$ are called the {\em source map} and {\em target map}
		respectively, and $m$ is called the {\em multiplication map}. Also elements of the submanifold $i\colon \Gamma_{0}
		\hookrightarrow \Gamma$ are called {\em units}. 
	\end{dfn}
	Let us consider the trivial vector bundle $\theta\colon \left(\mathfrak{n}_{\mathbf{m}_{l}^{l+1}}^{-}\right)^{\oplus l}\times \mathbb{D}(H)\rightarrow \mathbb{D}(H)$
	and introduce a Lie groupoid structure on this bundle.
	As both of the source and target maps, we employ the projection map $\theta$.
	We regard $\mathbb{D}(H)$ as a submanifold of $\left(\mathfrak{n}_{\mathbf{m}_{l}^{l+1}}^{-}\right)^{\oplus l}\times \mathbb{D}(H)$
	by the zero section $i\colon \mathbb{D}(H)\ni \mathbf{c}\mapsto (\mathbf{0},\mathbf{c})\in \left(\mathfrak{n}_{\mathbf{m}_{l}^{l+1}}^{-}\right)^{\oplus (l-1)}\times \mathbb{D}(H)$.
	Let us consider a 
	linear isomorphism
	\[
		\begin{array}{cccc}
		\phi_{\mathbf{c}}\colon &\left(\mathfrak{n}_{\mathbf{m}_{l}^{l+1}}^{-}\right)^{\oplus l}&\longrightarrow &N_{\mathbf{m}_{l}^{l+1}}^{-}(\widehat{\mathbb{C}[z]}(\mathbf{c})_{l-1})\\
		&(X_{i})_{i=0,1,\ldots,l-1}&\longmapsto&
		(E_{n}+X_{0})+\sum_{i=1}^{l-1}X_{i}(z-c_{0})(z-c_{1})\cdots(z-c_{i-1})
		\end{array}.
	\]
	Then we define the multiplication $m(X(\mathbf{c}),Y(\mathbf{c}))$
	of $X(\mathbf{c}):=((X_{i}),\mathbf{c})$ and  $Y(\mathbf{c}):=((Y_{i}),\mathbf{c}) 
	\in \left(\mathfrak{n}_{\mathbf{m}_{l}^{l+1}}^{-}\right)^{\oplus l}\times \mathbb{D}(H)$
	by
	\[
		m(X(\mathbf{c}),Y(\mathbf{c})):=\phi_{\mathbf{c}}^{-1}(\phi_{\mathbf{c}}((X_{i}))\cdot \phi_{\mathbf{c}}((Y_{i}))).
	\]
	Here $\cdot$ in the right hand side of the equation is the multiplication in $N_{\mathbf{m}_{l}^{l+1}}^{-}(\widehat{\mathbb{C}[z]}(\mathbf{c})_{l-1})$.
	Then  
	the tuple 
	\[
		(N_{\mathbf{m}_{l}^{l+1}}^{-})_{\mathbb{D}(H)}:=
	\left(\left(\mathfrak{n}_{\mathbf{m}_{l}^{l+1}}^{-}\right)^{\oplus l}\times \mathbb{D}(H),
	\mathbb{D}(H),\theta,\theta,m\right)
	\]
	becomes a complex Lie groupoid.
	Together with this Lie groupoid, 
	we consider the subbundle $\left(\mathfrak{n}_{\mathbf{m}_{l}^{l+1}}^{-}\right)^{\oplus (l-1)}\times \mathbb{D}(H)\rightarrow \mathbb{D}(H)$
	with the inclusion 
	\[
		\begin{array}{ccc}
		\left(\mathfrak{n}_{\mathbf{m}_{l}^{l+1}}^{-}\right)^{\oplus (l-1)}\times \mathbb{D}(H)&
		\longrightarrow  &\left(\mathfrak{n}_{\mathbf{m}_{l}^{l+1}}^{-}\right)^{\oplus l}\times \mathbb{D}(H)\\
		((X_{1},X_{2},\ldots,X_{l-1}),\mathbf{c})&\longmapsto& 
		((0,X_{1},X_{2},\ldots,X_{l-1}),\mathbf{c})
		\end{array}.
	\]
	Then this subbundle has the natural Lie groupoid structure whose 
	fiber at each $\mathbf{c}\in \mathbb{D}(H)$ is isomorphic 
	to $N_{\mathbf{m}_{l}^{l+1}}^{-}(\widehat{\mathbb{C}[z]}(\mathbf{c})_{l-1})^{1}$.
	We especially denote this Lie subgroupoid
	by $(N^{-})^{(l)}_{\mathbb{D}(H)}$.
	
	Next we introduce complex Lie algebroids.
	\begin{dfn}[Complex Lie algebroid]
		A {\em complex Lie algebroid} over a complex manifold $M$ is a holomorphic vector bundle 
		$A\rightarrow M$, together with a Lie bracket $[\cdot,\cdot]$ on its space of holomorphic 
		sections, such that there exists a vector bundle map
		$\mathbf{a}\colon A\rightarrow TM$ called the {\em anchor map} satisfying the Leibniz rule
		\[
			[\sigma,f\tau]=f[\sigma,\tau]+(\mathbf{a}(\sigma)f)\tau	
		\]
		for all global sections $\sigma,\tau$ of $A\rightarrow M$ and $f\in \mathcal{O}_M(M)$.
	\end{dfn}
	Let $\mathfrak{k}\subset M_{n}(\mathbb{C})$
	be a $\mathbb{C}$-subalgebra. We consider a linear isomorphism
	\[
		\begin{array}{cccc}
		\psi_{\mathbf{c}}\colon &\mathfrak{k}^{\oplus l}&\longrightarrow &\mathfrak{k}(\widehat{\mathbb{C}[z]}(\mathbf{c})_{l-1})\\
		&(X_{i})_{i=1,2\ldots,l-1}&\longmapsto&
		X_{0}+\sum_{i=1}^{l-1}X_{i}(z-c_{0})(z-c_{1})\cdots(z-c_{i-1})
		\end{array},
	\]
	and define the Lie bracket $[X(\mathbf{c}),Y(\mathbf{c})]$ 
	of $X(\mathbf{c}):=((X_{i}),\mathbf{c})$ and  $Y(\mathbf{c}):=((Y_{i}),\mathbf{c}) 
	\in \mathfrak{k}^{\oplus l}\times \mathbb{D}(H)$
	by
	\[
		[(X(\mathbf{c}),Y(\mathbf{c}))]:=\psi_{\mathbf{c}}^{-1}([\psi_{\mathbf{c}}((X_{i}), \psi_{\mathbf{c}}((Y_{i}))]).
	\]
	Then  the vector bundle $\theta \colon \mathfrak{k}^{\oplus l}\times \mathbb{D}(H)\rightarrow \mathbb{D}(H)$
	with the Lie bracket and the zero map as the trivial anchor map 
	becomes a Lie algebroid which we denote by $\mathfrak{k}_{\mathbb{D}(H)}$.
	The direct sum decomposition $\mathfrak{k}(\widehat{\mathbb{C}[z]}(\mathbf{c})_{l-1})
	=\mathfrak{k}\oplus \mathfrak{k}(\widehat{\mathbb{C}[z]}(\mathbf{c})_{l-1})^{1}$
	as Lie algebras induces the decomposition 
	$\mathfrak{k}_{\mathbb{D}(H)}=(\mathfrak{k}_{\mathbb{D}(H)})_{0}\oplus (\mathfrak{k}_{\mathbb{D}(H)})_{1}$
	as vector bundles.
	Here $(\mathfrak{k}_{\mathbb{D}(H)})_{0}$ and $(\mathfrak{k}_{\mathbb{D}(H)})_{1}$
	are trivial bundles 
	$\mathfrak{k}\times \mathbb{D}(H)\rightarrow \mathbb{D}(H)$
	and $\mathfrak{k}^{\oplus (l-1)}\times \mathbb{D}(H)\rightarrow \mathbb{D}(H)$
	with the Lie algebroid structures induced 
	form $\mathfrak{k}$ and $\mathfrak{k}(\widehat{\mathbb{C}[z]}(\mathbf{c})_{l-1})^{1}$.
	
	In particular, we consider the cases $\mathfrak{k}=\mathfrak{n}_{\mathbf{m}_{l}^{l+1}}^{-}$ and $\mathfrak{p}_{\mathbf{m}_{l}^{l+1}}^{-}$.
	In these cases, we can define the adjoint action of the Lie groupoid $(N_{\mathbf{m}_{l}^{l+1}}^{-})_{\mathbb{D}(H)}$ as follows.
	Let us take 
	holomorphic sections $n$ and $X$ of $(N_{\mathbf{m}_{l}^{l+1}}^{-})_{\mathbb{D}(H)}$ and $\mathfrak{k}_{\mathbb{D}(H)}$
	respectively. Then 
	we define the holomorphic section of $\mathfrak{k}_{\mathbb{D}(H)}$ by 
	\[
		\mathrm{Ad}_{\mathbb{D}(H)}(n)(X)(\mathbf{c}):=\psi_{\mathbf{c}}^{-1}\circ 
		\mathrm{Ad}_{N_{\mathbf{m}_{l}^{l+1}}^{-}(\widehat{\mathbb{C}[z]}(\mathbf{c})_{l-1})}(n(\mathbf{c}))(\psi_{\mathbf{c}}(X(\mathbf{c})))\quad (\mathbf{c}\in \mathbb{D}(H)).
	\]
	As a dual of this adjoint action, we can define  
	the coadjoint action of $(N_{\mathbf{m}_{l}^{l+1}}^{-})_{\mathbb{D}(H)}$
	on the dual bundle $\mathfrak{k}_{\mathbb{D}(H)}^{*}$.
	Namely, 
	for holomorphic sections $n$, $X$, and $\xi$ of $(N_{\mathbf{m}_{l}^{l+1}}^{-})_{\mathbb{D}(H)}$, $\mathfrak{k}_{\mathbb{D}(H)}$,
	and $\mathfrak{k}_{\mathbb{D}(H)}^{*}$, we define the 
	section of $\mathfrak{k}_{\mathbb{D}(H)}^{*}$ by
	\[
		(\mathrm{Ad}_{\mathbb{D}(H)}^{*}(n)(\xi))(X)(\mathbf{c}):=\xi(\mathbf{c})(\mathrm{Ad}_{N_{\mathbf{m}_{l}^{l+1}}^{-}(\widehat{\mathbb{C}[z]}(\mathbf{c})_{l-1})}(n(\mathbf{c})^{-1})(X(\mathbf{c})))
		\quad (\mathbf{c}\in \mathbb{D}(H)).
	\]
	We denote the subbundle $(\mathfrak{k}_{\mathbb{D}(H)})_{1}$
	of $\mathfrak{k}_{\mathbb{D}(H)}$ by 
	$(\mathfrak{n}^{-})^{(l)}_{\mathbb{D}(H)}$
	particularly when $\mathfrak{k}=\mathfrak{n}^{-}_{\mathbf{m}_{l}^{l+1}}$.
	
	The paring $(\ref{eq:paring})$ allows us to 
	identify the fiber of the dual bundle $\mathfrak{k}_{\mathbb{D}(H)}^{*}$
	at each $\mathbf{c}\in \mathbb{D}(H)$ with 
	$\mathfrak{n}^{+}_{\mathbf{m}_{l}^{l+1}}(\widehat{\varOmega}(\mathbf{c})_{l-1})$
	(resp. $\mathfrak{p}^{+}_{\mathbf{m}_{l}^{l+1}}(\widehat{\varOmega}(\mathbf{c})_{l-1})$)
	when $\mathfrak{k}=\mathfrak{n}^{-}_{\mathbf{m}_{l}^{l+1}}$
	(resp. $\mathfrak{k}=\mathfrak{p}^{-}_{\mathbf{m}_{l}^{l+1}}$).
	Thus we can regard $((\mathfrak{n}^{-})^{(l)}_{\mathbb{D}(H)})^{*}$
	as the subbundle of $(\mathfrak{p}^{-}_{\mathbf{m}_{l}^{l+1}})_{\mathbb{D}(H)}^{*}$
	by the inclusion  $\mathfrak{n}^{+}_{\mathbf{m}_{l}^{l+1}}(\widehat{\varOmega}(\mathbf{c})_{l-1})^{1}
	\subset \mathfrak{p}^{+}_{\mathbf{m}_{l}^{l+1}}(\widehat{\varOmega}(\mathbf{c})_{l-1})$
	at each fiber.
	
	\subsubsection{Symplectic foliation as deformation of $\mathbb{O}_{H_{\mathrm{irr}}}$}
	In Theorem \ref{thm:stepwise}, we saw that $\mathbb{O}_{H_{\mathrm{irr}}}$ is isomorphic to 
the product of $T^{*}(N^{-})^{(l)}$.
Thus deforming $T^{*}(N^{-})^{(l)}$, we can obtain a deformation of  $\mathbb{O}_{H_{\mathrm{irr}}}$ as follows.

Let us consider a relative cotangent bundle
\[
	\theta_{(N^{-})^{(l)}_{\mathbb{D}(H)}/\mathbb{D}(H)}
\colon T^{*}\left((N^{-})^{(l)}_{\mathbb{D}(H)}\right)/\mathbb{D}(H)
\rightarrow (N^{-})^{(l)}_{\mathbb{D}(H)},
\]
i.e., the vector bundle corresponding to 
the quotient sheaf $\varOmega_{(N^{-})^{(l)}_{\mathbb{D}(H)}}/
\theta^{*}\varOmega_{\mathbb{D}(H)}$.
Here $\theta\colon (N^{-})^{(l)}_{\mathbb{D}(H)}\rightarrow \mathbb{D}(H)$
is the Lie groupoid defined above and $\varOmega_{M}$ denotes the sheaf of 
holomorphic differential forms on a complex manifold $M$.
Set $\mathfrak{n}^{(l)}:=\left(\mathfrak{n}_{\mathbf{m}_{l}^{l+1}}\right)^{\oplus(l-1)}$
for simplicity.
Recall that if we forget the structure of Lie groupoid from $(N^{-})^{(l)}_{\mathbb{D}(H)}$,
this is the trivial bundle $\theta_{l}\colon \mathfrak{n}^{(l)}\times \mathbb{D}(H)\rightarrow \mathbb{D}(H)$.
Thus $T^{*}\left((N^{-})^{(l)}_{\mathbb{D}(H)}\right)/\mathbb{D}(H)$ is isomorphic to 
\[
	\theta_{\mathfrak{n}^{(l)}}\times \mathrm{id}_{\mathbb{D}(H)}\colon (T^{*}\mathfrak{n}^{(l)})
	\times \mathbb{D}(H)\rightarrow \mathfrak{n}^{(l)}\times \mathbb{D}(H),
\]
where $\theta_{\mathfrak{n}^{(l)}}\colon T^{*}\mathfrak{n}^{(l)}\rightarrow \mathfrak{n}^{(l)}$
is the standard projection of the cotangent bundle.
Here $T^{*}\mathfrak{n}^{(l)}$ has the Poisson structure 
induced from the standard symplectic form $\Omega_{T^{*}\mathfrak{n}^{(l)}}$
and also $\mathbb{D}(H)$ can be seen as the Poisson manifold with the trivial Poisson bracket.
Thus $T^{*}\left((N^{-})^{(l)}_{\mathbb{D}(H)}\right)/\mathbb{D}(H)$
becomes a Poisson manifold as
the product of these Poisson manifolds.
\begin{dfn}[Deformation of $T^{*}(N^{-})^{(l)}$]
	Let us set 
	\[
		\left(T^{*}(N^{-})^{(l)}\right)_{\mathbb{D}(H)}:=T^{*}\left((N^{-})^{(l)}_{\mathbb{D}(H)}\right)/\mathbb{D}(H)	
	\]
	and 
	call this Poisson manifold the {\em deformation of }$T^{*}(N^{-})^{(l)}$
	with respect to the HTL normal form $H$.
\end{dfn}
Let us define the projection map $\theta_{l}^{*}\colon \left(T^{*}(N^{-})^{(l)}\right)_{\mathbb{D}(H)}
\rightarrow \mathbb{D}(H)$ by setting $\theta_{l}^{*}:=\theta_{l}\circ \theta_{(N^{-})^{(l)}_{\mathbb{D}(H)}/\mathbb{D}(H)}$.
Then we note that $\theta_{l}^{*}\colon \left(T^{*}(N^{-})^{(l)}\right)_{\mathbb{D}(H)}
\rightarrow \mathbb{D}(H)$ is still a holomorphic vector bundle regarding $T^{*}\mathfrak{n}^{(l)}$
as a $\mathbb{C}$-vector space through the isomorphism $T^{*}\mathfrak{n}^{(l)}\cong \mathfrak{n}^{(l)}\oplus(\mathfrak{n}^{(l)})^{*}$.
For each fiber at $\mathbf{c}\in \mathbb{D}(H)$, there exists the natural isomorphism
\[
	(\theta_{l}^{*})^{-1}(\mathbf{c})\cong 
T^{*}N_{\mathbf{m}_{l}^{l+1}}^{-}(\widehat{\mathbb{C}[z]}(\mathbf{c})_{l-1})^{1}.
\]
Since in particular we have $(\theta_{l}^{*})^{-1}(0)\cong 
T^{*}N_{\mathbf{m}_{l}^{l+1}}^{-}(\mathbb{C}[z](0)_{l-1})^{1}=T^{*}(N^{-})^{(l)}$,
we can regard $\left(T^{*}(N^{-})^{(l)}\right)_{\mathbb{D}(H)}$ as a deformation of $T^{*}(N^{-})^{(l)}$.
It directly follows from the definition that 
the family $((\theta_{l}^{*})^{-1}(\mathbf{c}))_{\mathbf{c}\in \mathbb{D}(H)}$
is the symplectic foliation of $\left(T^{*}(N^{-})^{(l)}\right)_{\mathbb{D}(H)}$.

Since $N_{\mathbf{m}_{l}^{l+1}}^{-}(\widehat{\mathbb{C}[z]}(\mathbf{c})_{l-1})^{1}$
is normalized by $L_{\mathbf{m}_{1}}=\mathrm{GL}_{n}(\mathbb{C})_{H_{\mathrm{irr}}}$, the cotangent bundle 
$T^{*}N_{\mathbf{m}_{l}^{l+1}}^{-}(\widehat{\mathbb{C}[z]}(\mathbf{c})_{l-1})^{1}$
has the natural $L_{\mathbf{m}_{1}}$-action.
Hence $T^{*}(N^{-})^{(l)}_{\mathbb{D}(H)}$ is equipped with 
the $L_{\mathbf{m}_{1}}$-action.

Then as the direct sum of vector bundles $\theta_{l}^{*}\colon \left(T^{*}(N^{-})^{(l)}\right)_{\mathbb{D}(H)}
\rightarrow \mathbb{D}(H)$ we define a deformation of $\mathbb{O}_{H_{\mathrm{irr}}}$ as follows.
\begin{dfn}[Deformation of $\mathbb{O}_{H_{\mathrm{irr}}}$]
We consider the direct sum of the vector bundles 
\[
	\theta_{\mathbb{O}_{H_{\mathrm{irr}}}}\colon \bigoplus_{l=2}^{k}\left(T^{*}(N^{-})^{(l)}\right)_{\mathbb{D}(H)}\rightarrow \mathbb{D}(H)	
\]
and we denote the total space of this bundle 
by $(\mathbb{O}_{H_{\mathrm{irr}}})_{\mathbb{D}(H)}$,
which is called the {\em deformation of }$\mathbb{O}_{H_{\mathrm{irr}}}$.
This complex manifold
$(\mathbb{O}_{H_{\mathrm{irr}}})_{\mathbb{D}(H)}$
is naturally equipped with the Poisson structure and $L_{\mathbf{m}_{1}}$-action as well as $\left(T^{*}(N^{-})^{(l)}\right)_{\mathbb{D}(H)}$.
\end{dfn}

\begin{prp}
	The family $(\theta_{\mathbb{O}_{H_{\mathrm{irr}}}}^{-1}(\mathbf{c}))_{\mathbf{c}\in \mathbb{D}(H)}$
	is the symplectic foliation of $(\mathbb{O}_{H_{\mathrm{irr}}})_{\mathbb{D}(H)}$.
	In particular, the special fiber at $\mathbf{c}=0$, we have the symplectic isomorphism
	\[
		\theta_{\mathbb{O}_{H_{\mathrm{irr}}}}^{-1}(0)\cong \mathbb{O}_{H_{\mathrm{irr}}}.
	\]
\end{prp}
\begin{proof}
	The first assertion is obvious since the family $((\theta_{l}^{*})^{-1}(\mathbf{c}))_{\mathbf{c}\in \mathbb{D}(H)}$
	is the symplectic foliation of $\left(T^{*}(N^{-})^{(l)}\right)_{\mathbb{D}(H)}$
	for each $l$.
	The second assertion follows from Theorem \ref{thm:stepwise}. 
\end{proof}
  
The right trivialization
\[
	\begin{array}{ccc}
		N_{\mathbf{m}_{l}^{l+1}}^{-}(\widehat{\mathbb{C}[z]}(\mathbf{c})_{l-1})^{1}\times 
		(\mathfrak{n}_{\mathbf{m}_{l}^{l+1}}^{-}(\widehat{\mathbb{C}[z]}(\mathbf{c})_{l-1})^{1})^{*}&
\longrightarrow& T^{*}N_{\mathbf{m}_{l}^{l+1}}^{-}(\widehat{\mathbb{C}[z]}(\mathbf{c})_{l-1})^{1}\\
(n,\xi)&\longmapsto &R^{*}_{n^{-1}}(\xi)
	\end{array}
\]
gives 
the isomorphism 
\begin{equation}\label{eq:rigthtrivbundle}
	(N^{-})^{(l)}_{\mathbb{D}(H)}\oplus 
	\left((\mathfrak{n}^{-})^{(l)}_{\mathbb{D}(H)}\right)^{*}
	\cong 
	\left(T^{*}(N^{-})^{(l)}\right)_{\mathbb{D}(H)}.
\end{equation}
as vector budless. Thus we have the isomorphism 
\begin{equation}\label{eq:deform}
	(\mathbb{O}_{H_{\mathrm{irr}}})_{\mathbb{D}(H)}\cong \bigoplus_{l=2}^{k}\left((N^{-})^{(l)}_{\mathbb{D}(H)}\oplus 
	\left((\mathfrak{n}^{-})^{(l)}_{\mathbb{D}(H)}\right)^{*}\right).	
\end{equation}

	\subsection{Local unfolding manifold}
	In the previous section, we defined a deformation of $\mathbb{O}_{H_{\mathrm{irr}}}$.
	Now we shall construct a deformation of the truncated orbit $\mathbb{O}_{H}$.

	\subsubsection{Unfolding manifold of a truncated orbit}
	As we saw in Theorem \ref{thm:boalch}, truncated orbit $\mathbb{O}_{H}$
	is obtained as a symplectic reduction of the extended orbit $T^{*}\mathrm{GL}_{n}(\mathbb{C})
	\times \mathbb{O}_{H_{\mathrm{irr}}}$. 
	Based on this fact, 
	we shall define a deformation of the truncated orbit 
	as a Poisson reduction of 
	$T^{*}\mathrm{GL}_{n}(\mathbb{C})\times (\mathbb{O}_{H_{\mathrm{irr}}})_{\mathbb{D}(H)}$.
	To this purpose, we first consider a deformation of the 
	map $\mathrm{res}_{\mathbb{O}_{H_{\mathrm{irr}}}}\colon 
	\mathbb{O}_{H_{\mathrm{irr}}}\rightarrow \mathfrak{gl}_{n}(\mathbb{C})_{H_{\mathrm{irr}}}^{*}$ as follows.
	Let $\iota_{l} \colon \left((\mathfrak{n}^{-})^{(l)}_{\mathbb{D}(H)}\right)^{*}
	\hookrightarrow (\mathfrak{p}_{\mathbf{m}_{l}^{l+1}}^{-})^{*}_{\mathbb{D}(H)}$
	be the inclusion as subbundle.
	Under the identification $(\ref{eq:deform})$, we define the map 
	$\mathrm{res}_{\mathfrak{l}_{\mathbf{m}_{1}}}\colon (\mathbb{O}_{H_{\mathrm{irr}}})_{\mathbb{D}(H)}\rightarrow \mathfrak{l}_{\mathbf{m}_{1}}^{*}$
	by setting 
	\[
		\mathrm{res}_{\mathfrak{l}_{\mathbf{m}_{1}}}(n_{l}(\mathbf{c}),X_{l}(\mathbf{c}))_{l=2,3,\ldots,k}:=
		\pi_{(\mathfrak{p}_{\mathbf{m}_{l}^{l+1}}^{-})^{*}\downarrow \mathfrak{l}_{\mathbf{m}_{1}}^{*}}\left(\sum_{l=2}^{k}-\mathrm{res}_{x=\infty}\left((\mathrm{Ad}^{*}(n_{l}(\mathbf{c}))(\iota_{l}(X_{l}(\mathbf{c}))))\right)\right).
	\]
	for $(n_{l}(\mathbf{c}),X_{l}(\mathbf{c}))_{l=2,3,\ldots,k}
	\in \bigoplus_{l=2}^{k}\left(N_{\mathbf{m}_{l}^{l+1}}^{-}(\mathbb{D}(H))^{1}\oplus 
	\left(\mathfrak{n}_{\mathbf{m}_{l}^{l+1}}^{-}(\mathbb{D}(H))^{1}\right)^{*}\right)$.
	In the right hand side, $\pi_{(\mathfrak{p}_{\mathbf{m}_{l}^{l+1}}^{-})^{*}\downarrow \mathfrak{l}_{\mathbf{m}_{1}}^{*}}\colon 
	(\mathfrak{p}_{\mathbf{m}_{l}^{l+1}}^{-})^{*}\rightarrow \mathfrak{l}_{\mathbf{m}_{1}}^{*}$
	denotes the natural projection.
	Then we obtain the following.
	\begin{prp}
		\begin{enumerate}
		\item 
		The map $\mathrm{res}_{\mathfrak{l}_{\mathbf{m}_{1}}}$ is 
		$L_{\mathbf{m}_{1}}=\mathrm{GL}_{n}(\mathbb{C})_{H_{\mathrm{irr}}}$-equivariant.
		\item The map $\mathrm{res}_{\mathfrak{l}_{\mathbf{m}_{1}}}$ is a deformation of $\mathrm{res}_{\mathbb{O}_{H_{\mathrm{irr}}}}$.
		Namely, under the identification $\theta_{\mathbb{O}_{H_{\mathrm{irr}}}}^{-1}(0)\cong \mathbb{O}_{H_{\mathrm{irr}}}$, we have 
		\[
			\mathrm{res}_{\mathfrak{l}_{\mathbf{m}_{1}}}|_{\theta_{\mathbb{O}_{H_{\mathrm{irr}}}}^{-1}(0)}=\mathrm{res}_{\mathbb{O}_{H_{\mathrm{irr}}}}.
		\]
		\end{enumerate}
	\end{prp}
	\begin{proof}
		The first assertion follows from the definition and  
	second one from Proposition \ref{prop:stepwise}. 
	\end{proof}
	Let us consider the map 
	\[
		\begin{array}{cccc}	
			\mu_{\mathrm{ext},\,\mathbb{D}(H)}\colon &T^{*}\mathrm{GL}_{n}(\mathbb{C})\times (\mathbb{O}_{H_{\mathrm{irr}}})_{\mathbb{D}(H)}
			&\rightarrow& \mathfrak{gl}_{n}(\mathbb{C})_{H_{\mathrm{irr}}}^{*}\\
			&(\alpha,\xi)&\longmapsto &\iota_{H_{\mathrm{irr}}}^{*}\circ \mu_{\mathrm{GL}_{n}(\mathbb{C})}(\alpha)
			+\mathrm{res}_{\mathfrak{l}_{\mathbf{m}_{1}}}(\xi)
		\end{array}.
	\]
	Then we can see that this map is a $\mathrm{GL}_{n}(\mathbb{C})_{H_{\mathrm{irr}}}$-equivariant moment map.
	\begin{dfn}[Local unfolding manifold]
		The Poisson reduction
		\[
			(\mathbb{O}_{H})_{\mathbb{D}(H)}:=(\mathrm{GL}_{n}(\mathbb{C})_{H_{\mathrm{irr}}})_{H_{\mathrm{res}}}
			\backslash 	\mu_{\mathrm{ext},\,\mathbb{D}(H)}^{-1}(H_{\mathrm{res}})
		\]
		is called the {\em local unfolding manifold} of the truncated orbit $\mathbb{O}_{H}$.
		We have the projection map $\theta_{\mathbb{D}(H)}\colon (\mathbb{O}_{H})_{\mathbb{D}(H)}
		\rightarrow \mathbb{D}(H)$ induced from the projection $\theta_{\mathbb{O}_{H_{\mathrm{irr}}}}
		\colon (\mathbb{O}_{H_{\mathrm{irr}}})_{\mathbb{D}(H)}
		\rightarrow \mathbb{D}(H)$.
	\end{dfn}
	It is easy to see that the $\mathrm{GL}_{n}(\mathbb{C})_{H_{\mathrm{irr}}}$-action 
	on $T^{*}\mathrm{GL}_{n}(\mathbb{C})\times (\mathbb{O}_{H_{\mathrm{irr}}})_{\mathbb{D}(H)}$ is 
	proper. Moreover  
	a similar argument as on the symplectic reduction shows that 
	$H_{\mathrm{res}}$ is a regular value of the map $\mu_{\mathrm{ext},\,\mathbb{D}(H)}$.
	Thus the local unfolding manifold $(\mathbb{O}_{H})_{\mathbb{D}(H)}$ is truly a complex manifold.
	
	\subsubsection{Symplecitc foliation of $(\mathbb{O}_{H})_{\mathbb{D}(H)}$}
	Take a partition
	$\mathcal{I}\colon I_{0},I_{1},\ldots,I_{r} \in \mathcal{P}_{[k+1]}$ 
	and
	the corresponding embedding $\iota_{\mathcal{I}}\colon \mathbb{C}^{r+1}\hookrightarrow \mathbb{C}^{k+1}$
	as before.
	Then 
	we can consider the deformation $H_{\mathcal{I}}(\mathsf{c})=H(\iota_{\mathcal{I}}(\mathsf{c}))$ 
	of $H$ on 
	$C(\mathcal{I})=\{\iota_{\mathcal{I}}(\mathsf{c})\in \mathbb{C}^{k+1}\mid \mathsf{c}=(c_{0},c_{1},\ldots,c_{r})\in C_{r+1}(\mathbb{C})\}$.
	Then as we saw in Proposition \ref{prop:spectral type}, the partial fraction decomposition algorithm 
	gives a description of 
	$H_{\mathcal{I}}(\mathsf{c})$ as the sum of 
	HLT normal forms $H_{\mathcal{I}}(\mathsf{c})_{z_{c_{j}}}$ at $z=c_{j}$
	with spectral types 
	\[
		\begin{cases}
		(\mathbf{m}_{i_{[j,0]}}\le_{\phi_{i_{[j,0]}}^{i_{[j,1]}}}\mathbf{m}_{i_{[j,1]}}\le_{\phi_{i_{[j,1]}}^{i_{[j,2]}}}\cdots\le_{\phi_{i_{[j,k_{j}-1]}}^{i_{[j,k_{j}]}}} \mathbf{m}_{i_{[j,k_{j}]}},\mathrm{triv})&\text{ for }
	j=1,2,\ldots,r,\\
		(\mathbf{m}_{i_{[0,0]}}\le_{\phi_{i_{[0,0]}}^{i_{[0,1]}}}\mathbf{m}_{i_{[0,1]}}\le_{\phi_{i_{[0,1]}}^{i_{[0,2]}}}\cdots\le_{\phi_{i_{[0,k_{0}-1]}}^{i_{[0,k_{0}]}}} \mathbf{m}_{i_{[0,k_{0}]}},\sigma(\mathbf{m}_{0}))&\text{ for }
	j=0
		\end{cases}.
	\]
	We denote the coadjoint orbit of $\mathrm{GL}_{n}(\mathbb{C}[z_{c_{j}}]_{k_{j}})$
	through $H_{\mathcal{I}}(\mathsf{c})_{z_{c_{j}}}$ by $\mathbb{O}_{H_{\mathcal{I}}(\mathsf{c})_{z_{c_{j}}}}$.
	
	As we saw in Section \ref{sec:pfd}, the partial fraction decomposition of $\widehat{\varOmega}(\mathbf{c})_{l}$ 
	is compatible with the decomposition of $\widehat{\mathbb{C}[z]}(\mathbf{c})_{l}$.
	Thus this induces the partial fraction decomposition of the each fiber of $T^{*}N^{(l)}$
	through the trivialization $(\ref{eq:rigthtrivbundle})$, and we obtain a symplectic holomorphic map 
	\[
		\varphi_{\iota_{\mathcal{I}}(\mathsf{c})}
		\colon \theta_{\mathbb{D}(H)}^{-1}(\iota_{\mathcal{I}}(\mathsf{c}))
		\hookrightarrow \prod_{j=0}^{r}\mathbb{O}_{H_{\mathcal{I}}(\mathsf{c})_{z_{c_{j}}}}.
	\]
	Then we obtain the following.
	\begin{thm}[\cite{H}]\label{thm:fiber}
		\begin{enumerate}
			\item The holomorphic family $(\theta_{\mathbb{D}(H)}^{-1}(\mathbf{c}))_{\mathbf{c}\in 
			\mathbb{D}(H)}$ is the symplectic foliation of $(\mathbb{O}_{H})_{\mathbb{D}(H)}$.
			\item Let 
			$\mathcal{I}\colon I_{0},I_{1},\ldots,I_{r} \in \mathcal{P}_{[k+1]}$ 
			be a partition. Let us take  $\mathsf{c}=(c_{0},c_{1},\ldots,c_{r})\in C_{r+1}(\mathbb{C})$
			so that 
			$\iota_{\mathcal{I}}(\mathsf{c})\in \mathbb{D}(H)$.
			Then the partial fraction decomposition explained in Section \ref{sec:pfd}
			gives the holomorphic symplectic embedding  
			\[
				\varphi_{\iota_{\mathcal{I}}(\mathsf{c})}
				\colon \theta_{\mathbb{D}(H)}^{-1}(\iota_{\mathcal{I}}(\mathsf{c}))
				\hookrightarrow \prod_{j=0}^{r}\mathbb{O}_{H_{\mathcal{I}}(\mathsf{c})_{z_{c_{j}}}}
			\]
			whose image is a Zariski open dense subset.
		\end{enumerate}
	\end{thm}
	This shows that 
	the symplectic foliation of the unfolding manifold $(\mathbb{O}_{H})_{\mathbb{D}(H)}$
	gives a deformation of the truncated orbit $\mathbb{O}_{H}$
	along with the deformation $(H(\mathbf{c}))_{\mathbf{c}\in \mathbb{D}(H)}$
	of the HTL normal form $H$
	introduced in Definition \ref{def:unfold}.
	
	Let us introduce a deformation of the residue map $\mu_{\mathbb{O}_{H}}^{0}\colon 
	\mathbb{O}_{H}\rightarrow \mathfrak{gl}_{n}(\mathbb{C})^{*}$ of the truncated orbit 
	defined in Section \ref{sec:red} as follows.
	Consider the map  
	\[
		\nu_{\mathrm{GL}_{n}(\mathbb{C})}\circ \mathrm{pr}_{1}\colon T^{*}\mathrm{GL}_{n}(\mathbb{C})\times (\mathbb{O}_{H_{\mathrm{irr}}})_{\mathbb{D}(H)}\rightarrow 
		\mathfrak{gl}_{n}(\mathbb{C})^{*}.
	\]
	Then as well as in Section \ref{sec:red}, we can see that 
	its restriction 
	$\nu_{\mathrm{GL}_{n}(\mathbb{C})}\circ \mathrm{pr}_{1}|_{\mu_{\mathrm{ext},\mathbb{D}(H)}^{-1}(H_{\mathrm{res}})}$
	on the subspace $\mu_{\mathrm{ext},\mathbb{D}(H)}^{-1}(H_{\mathrm{res}})$
	is $(\mathrm{GL}_{n}(\mathbb{C})_{H_{\mathrm{irr}}})_{H_{\mathrm{res}}}$-invariant.
	Thus there uniquely exists the map 
	$\nu_{(\mathbb{O}_{H})_{\mathbb{D}(H)}}\colon (\mathbb{O}_{H})_{\mathbb{D}(H)}\rightarrow \mathfrak{gl}_{n}(\mathbb{C})^{*}$ such that 
	the diagram 
	\[
		\begin{tikzcd}
			\mu_{\mathrm{ext},\mathbb{D}(H)}^{-1}(H_{\mathrm{res}}) \arrow[r,"\nu_{\mathrm{GL}_{n}(\mathbb{C})}\circ \mathrm{pr}_{1}"] \arrow[d]&
			[2 em]\mathfrak{gl}_{n}(\mathbb{C})^{*}\\
			(\mathbb{O}_{H})_{\mathbb{D}(H)}\arrow[ur, "\nu_{(\mathbb{O}_{H})_{\mathbb{D}(H)}}"']&
		\end{tikzcd}
	\]
	is commutative.
	Then Proposition \ref{prop:red} assures that the map 
	\[
		\mu_{(\mathbb{O}_{H})_{\mathbb{D}(H)}}^{0}:=-\nu_{(\mathbb{O}_{H})_{\mathbb{D}(H)}}\colon (\mathbb{O}_{H})_{\mathbb{D}(H)}\rightarrow \mathfrak{gl}_{n}(\mathbb{C})^{*}
	\]
	is a deformation of the map $\mu_{\mathbb{O}_{H}}^{0}\colon \mathbb{O}_{H}\rightarrow \mathfrak{gl}_{n}(\mathbb{C})^{*}$, i.e., we have the equation
	\[
		\mu_{(\mathbb{O}_{H})_{\mathbb{D}(H)}}^{0}|_{\theta_{\mathbb{D}(H)}^{-1}(0)}=\mu_{\mathbb{O}_{H}}^{0}
	\]
	under the isomorphism $\theta_{\mathbb{D}(H)}^{-1}(0)\cong \mathbb{O}_{H}$.
	
	\subsection{Unfolding manifold of meromorphic connections}
	Now we are ready to define the unfolding manifold which is a Poisson manifold 
	descending the unfolding of the unramified irregular singularities of 
	connections in the moduli space $\mathcal{M}_{s}^{*}(\mathbf{H})$. 
	
	\subsubsection{Stratification of the product of local unfolding manifolds}
	Let $\mathbf{H}=(H_{a})_{a\in |D|}$
	be a collection of unramified HTL normal forms,
	\[
	H_{a}=\left(\frac{S_{k_{a}}^{(a)}}{z_{a}^{k_{a}-1}}+\cdots+\frac{S^{(a)}_{1}}{z_{a}}+S^{(a)}_{0}+N^{(a)}_{0}\right)
	\frac{dz_{a}}{z_{a}}
	\]
	with the spectral types $\mathrm{sp}(H_{a}):=(\mathbf{m}_{0}^{[a]}\le_{\phi_{0}^{[a]}}\le  \mathbf{m}^{[a]}_{1}\le_{\phi_{1}^{[a]}} \cdots \le_{\phi_{k-1}^{[a]}}\mathbf{m}_{k_{a}}^{[a]}, \sigma(\mathbf{m}_{0}^{[a]}))$.
	We consider the moduli space 
	$\mathcal{M}_{s}^{*}(\mathbf{H})=\mathrm{GL}_{n}(\mathbb{C})\backslash (\mu_{\mathbf{H}}^{s})^{-1}(0)$
	of meromorphic connections with respect to $\mathbf{H}$.
	Then we define a deformation of $\mathcal{M}_{s}^{*}(\mathbf{H})$ as follows.
	Let us consider the product of the local unfolding manifolds
	\[
		\prod_{a\in |D|}(\mathbb{O}_{H_{a}})_{\mathbb{D}(H_{a})}
	\]
	with the natural projection $\bar{\theta}_{\prod_{a\in |D|}\mathbb{D}(H_{a})}\colon \prod_{a\in |D|}(\mathbb{O}_{H_{a}})_{\mathbb{D}(H_{a})}\rightarrow\prod_{a\in |D|}\mathbb{D}(H_{a})$.
	
	As well as in Section \ref{sec:strata}, we can consider a stratification 
	\[
		\prod_{a\in |D|}\mathbb{C}^{k_{a}+1}=\bigsqcup_{(\mathcal{I}_{a})_{a
	\in |D|}\in \prod_{a\in |D|}\mathcal{P}_{k_{a}+1}}\left(\prod_{a\in |D|}C(\mathcal{I}_{a})	\right).
	\]
	of $\prod_{a\in |D|}\mathbb{C}^{k_{a}+1}$ associated to collections of partitions $(\mathcal{I}_{a})_{a
	\in |D|}\in \prod_{a\in |D|}\mathcal{P}_{k_{a}+1}$.
	This induces the following stratification of $\prod_{a\in |D|}(\mathbb{O}_{H_{a}})_{\mathbb{D}(H_{a})}$,
	\[
		\prod_{a\in |D|}(\mathbb{O}_{H_{a}})_{\mathbb{D}(H_{a})}=\bigsqcup_{(\mathcal{I}_{a})_{a
		\in |D|}\in \prod_{a\in |D|}\mathcal{P}_{k_{a}+1}}\prod_{a\in |D|}(\mathbb{O}_{H_{a}})_{\mathbb{D}(H_{a})}^{\mathcal{I}_{a}},
	\]
	where 
	\[
		(\mathbb{O}_{H_{a}})_{\mathbb{D}(H_{a})}^{\mathcal{I}_{a}}:=\bigcup_{\mathbf{c}_{a}\in \mathbb{D}(H)\cap C(\mathcal{I}_{a})}\theta_{\mathbb{D}(H_{a})}^{-1}(\mathbf{c}_{a}).
	\]
	Let us look at a stratum $\prod_{a\in |D|}(\mathbb{O}_{H_{a}})_{\mathbb{D}(H_{a})}^{\mathcal{I}_{a}}$. Then
	Theorem \ref{thm:fiber} says that 
	each fiber $\prod_{a\in |D|}\theta_{\mathbb{D}(H_{a})}^{-1}(\iota_{\mathcal{I}_{a}}(\mathsf{c}_{a}))$ at $(\mathsf{c}_{a})\in \prod_{a\in |D|}
	\mathbb{C}^{r_{\mathcal{I}_{a}}+1}$
	is isomorphic to a Zariski open subset of $\prod_{a\in |D|}\prod_{j=0}^{r_{\mathcal{I}_{a}}}\mathbb{O}_{H_{\mathcal{I}_{a}}(\mathsf{c}_{a})_{z_{c_{j}}}}$.
	\if0
	Then the spectral type of each $H_{\mathcal{I}_{a}}(\mathsf{c}_{a})_{z_{c_{j}}}$ is 
	\[
		(\mathbf{m}^{[a]}_{i^{[a]}_{[j,0]}}\le_{(\phi^{[a]})_{i^{[a]}_{[j,0]}}^{i^{[a]}_{[j,1]}}}\mathbf{m}^{[a]}_{i^{[a]}_{[j,1]}}\le_{(\phi^{[a]})_{i^{[a]}_{[j,1]}}^{i^{[a]}_{[j,2]}}}\cdots\le_{(\phi^{[a]})_{i^{[a]}_{[j,k_{a,j}-1]}}^{i^{[a]}_{[j,k_{a,j}]}}} \mathbf{m}^{[a]}_{i^{[a]}_{[j,k_{a,j}]}},\mathrm{triv})
	\]
	if  $j=1,2,\ldots,r_{\mathcal{I}_{a}}$, or 
	\[	
	(\mathbf{m}^{[a]}_{i^{[a]}_{[0,0]}}\le_{(\phi^{[a]})_{i^{[a]}_{[0,0]}}^{i^{[a]}_{[0,1]}}}\mathbf{m}^{[a]}_{i^{[a]}_{[0,1]}}\le_{(\phi^{[a]})_{i^{[a]}_{[0,1]}}^{i^{[a]}_{[0,2]}}}\cdots\le_{(\phi^{[a]})_{i^{[a]}_{[0,k_{a,0}-1]}}^{i^{[a]}_{[0,k_{a,0}]}}} \mathbf{m}^{[a]}_{i^{[a]}_{[0,k_{a,0}]}},\sigma(\mathbf{m}^{[a]}_{0}))
	\]
	if $j=0$.
	and write $\mathcal{I}_{a}=(I^{[a]}_{0},I^{[a]}_{1},\ldots,I^{[a]}_{r_{\mathcal{I}_{a}}})\in \mathcal{P}_{k_{a}+1}$, 
	$I^{[a]}_{j}=\{i^{[a]}_{[j,0]},i^{[a]}_{[j,1]},\ldots,i^{[a]}_{[j,k_{a,j}]}\}$.
	Here we note that these spectral types do not depend on the parameters $(\mathsf{c}_{a})_{a\in |D|}$ but only on the collection of partitions $(\mathcal{I}_{a})_{a\in |D|}$.
	Thus we just write 
	\[
		\mathrm{sp}(\mathbf{H})_{(\mathcal{I}_{a})_{a\in |D|}}:=(\mathrm{sp}(H_{\mathcal{I}_{a}}(\mathsf{c}_{a})_{z_{c_{j}}}))_{a\in |D|}.
	\]
	\fi
	
	By these isomorphisms, 
	we can define an open subspace 
	$\left(\prod_{a\in |D|}(\mathbb{O}_{H_{a}})_{\mathbb{D}(H_{a})}^{\mathcal{I}_{a}}\right)^{s}$ of the stratum
	as the union of the inverse images of $\left(\prod_{a\in |D|}\prod_{j=0}^{r_{\mathcal{I}_{a}}}\mathbb{O}_{H_{\mathcal{I}_{a}}(\mathsf{c}_{a})_{z_{c_{j}}}}\right)^{s}$.
	Further we define 
	\[
		\left(\prod_{a\in |D|}(\mathbb{O}_{H_{a}})_{\mathbb{D}(H_{a})}\right)^{s}:=\bigsqcup_{(\mathcal{I}_{a})_{a
		\in |D|}\in \prod_{a\in |D|}\mathcal{P}_{k_{a}+1}}\left(\prod_{a\in |D|}(\mathbb{O}_{H_{a}})_{\mathbb{D}(H_{a})}^{\mathcal{I}_{a}}\right)^{s}.
	\]
	
	\subsubsection{Unfolding manifold of meromorphic connections}
	Let us consider the $\mathrm{GL}_{n}(\mathbb{C})$-equivariant map 
	\[
		\begin{array}{cccc}
		\mu_{\mathbf{H},\prod_{a\in |D|}\mathbb{D}(H_{a})}^{s}\colon &
		\left(\prod_{a\in |D|}(\mathbb{O}_{H_{a}})_{\mathbb{D}(H_{a})}\right)^{s}
		&\longrightarrow &\mathfrak{gl}_{n}(\mathbb{C})^{*}\\
		&
		(X_{a})_{a\in |D|}&\longmapsto &
		\sum_{a\in |D|}\mu^{0}_{(\mathbb{O}_{H_{a}})_{\mathbb{D}(H_{a})}}(X_{a})
		\end{array}.
	\]
	\begin{dfn}[Unfolding manifold]
		The Poisson quotient space 
		\[
			\mathcal{M}_{s}^{*}(\mathbf{H})_{\prod_{a\in |D|}\mathbb{D}(H_{a})}:=
			\mathrm{GL}_{n}(\mathbb{C})\backslash 	(\mu_{\mathbf{H},\prod_{a\in |D|}\mathbb{D}(H_{a})}^{s})^{-1}(0)
		\] 
		is called the {\em unfolding manifold} of 
		meromorphic connections on the rank $n$ trivial bundle on $\mathbb{P}^{1}$
		with respect to the unramifed HTL normal forms $\mathbf{H}$.
	\end{dfn}
	Since 
	the projection map $\bar{\theta}_{\prod_{a\in |D|}\mathbb{D}(H_{a})}\colon \prod_{a\in |D|}(\mathbb{O}_{H_{a}})_{\mathbb{D}(H_{a})}\rightarrow\prod_{a\in |D|}\mathbb{D}(H_{a})$
	is  $\mathrm{GL}_{n}(\mathbb{C})$-invariant, 
	this map naturally induces the projection map 
	\[
		\theta_{\prod_{a\in |D|}\mathbb{D}(H_{a})}\colon \mathcal{M}_{s}^{*}(\mathbf{H})_{\prod_{a\in |D|}\mathbb{D}(H_{a})}\rightarrow\prod_{a\in |D|}\mathbb{D}(H_{a}).	
	\]
	Then we can show that the closed submanifold $	(\mu_{\mathbf{H},\prod_{a\in |D|}\mathbb{D}(H_{a})}^{s})^{-1}(0)$
	of the Poisson manifold $\left(\prod_{a\in |D|}(\mathbb{O}_{H_{a}})_{\mathbb{D}(H_{a})}\right)^{s}$
	has a clean intersection with both the symplectic leaves $\bar{\theta}_{\prod_{a\in |D|}\mathbb{D}(H_{a})}^{-1}(\mathbf{c})$
	of $\left(\prod_{a\in |D|}(\mathbb{O}_{H_{a}})_{\mathbb{D}(H_{a})}\right)^{s}$
	and the orbits of $\mathrm{GL}_{n}(\mathbb{C})$.
	Therefore we obtain 
	the following our first main theorem.
	\begin{thm}[\cite{H}]\label{thm:main1}
		Suppose that the moduli space $\mathcal{M}_{s}^{*}(\mathbf{H})$ is not 
		the empty set.
		Then the  unfolding manifold $\mathcal{M}_{s}^{*}(\mathbf{H})_{\prod_{a\in |D|}\mathbb{D}(H_{a})}$
		is a Poisson manifold with the symplectic foliation 
		$(\theta_{\prod_{a\in |D|}\mathbb{D}(H_{a})}^{-1}(\mathbf{c}))_{\mathbf{c}\in \prod_{a\in |D|}\mathbb{D}(H_{a})}.$
	\end{thm}
	Let us define a subspace 
	\begin{align*}
		\mathcal{M}_{s}^{*}(\mathbf{H})_{\prod_{a\in |D|}\mathbb{D}(H_{a})}^{(\mathcal{I}_{a})_{a
		\in |D|}}&:=\prod_{a\in |D|}\bigcup_{\mathbf{c}_{a}\in \mathbb{D}(H)\cap C(\mathcal{I}_{a})}\theta_{\mathbb{D}(H_{a})}^{-1}(\mathbf{c}_{a})\\
		&=\mathcal{M}_{s}^{*}(\mathbf{H})_{\prod_{a\in |D|}\mathbb{D}(H_{a})}\cap 
		\prod_{a\in |D|}(\mathbb{O}_{H_{a}})_{\mathbb{D}(H_{a})}^{\mathcal{I}_{a}}
	\end{align*}
	for each $(\mathcal{I}_{a})_{a
	\in |D|}\in \prod_{a\in |D|}\mathcal{P}_{k_{a}+1}$.
	Then we obtain a stratification 
	\[
		\mathcal{M}_{s}^{*}(\mathbf{H})_{\prod_{a\in |D|}\mathbb{D}(H_{a})}=
		\bigsqcup_{(\mathcal{I}_{a})_{a
		\in |D|}\in \prod_{a\in |D|}\mathcal{P}_{k_{a}+1}}\mathcal{M}_{s}^{*}(\mathbf{H})_{\prod_{a\in |D|}\mathbb{D}(H_{a})}^{(\mathcal{I}_{a})_{a
		\in |D|}}.
	\]
	
	Let us fix a collection $(\mathcal{I}_{a})_{a\in |D|}$ of partitions and 
	consider the corresponding stratum $\mathcal{M}_{s}^{*}(\mathbf{H})_{\prod_{a\in |D|}\mathbb{D}(H_{a})}^{(\mathcal{I}_{a})_{a
	\in |D|}}$.
	Then the fibers $\theta_{\prod_{a\in |D|}\mathbb{D}(H_{a})}^{-1}(\mathbf{c})$ contained in this stratum
	are of the forms $\theta_{\prod_{a\in |D|}\mathbb{D}(H_{a})}^{-1}((\iota_{\mathcal{I}_{a}}(\mathsf{c}_{a}))_{a\in |D|}))$
	with some $(\mathsf{c}_{a})_{a\in |D|}\in \prod_{a\in |D|}\mathbb{C}^{r_{\mathcal{I}_{a}}+1}$
	We associate the collection of HTL normal forms 
	\[
		\mathbf{H}((\mathsf{c}_{a})_{a\in |D|})^{(\mathcal{I}_{a})_{a\in |D|}}:=
		(H_{\mathcal{I}_{a}}(\mathsf{c}_{a})_{z_{c^{[a]}_{j}}})_{a\in |D|,\,0\le j\le r_{\mathcal{I}_{a}}}
	\]
	to the fiber $\theta_{\prod_{a\in |D|}\mathbb{D}(H_{a})}^{-1}((\iota_{\mathcal{I}_{a}}(\mathsf{c}_{a}))_{a\in |D|}))$.
	Then as a corollary of Theorem \ref{thm:fiber}, we obtain the second main theorem which 
	says that each symplectic leaf $\theta_{\prod_{a\in |D|}\mathbb{D}(H_{a})}^{-1}(\mathbf{c})$ of the unfolding manifold
	is embedded into a Zariski open subset of a moduli space of connections, namely, 
	the symplectic leaves describe a deformation of the moduli space of connections and moreover this deformation 
	corresponds to the unfolding of the irregular singularities of the associated HTL normal forms to the moduli space.
	\begin{thm}[\cite{H}]\label{thm:main2}
		In a stratum $\mathcal{M}_{s}^{*}(\mathbf{H})_{\prod_{a\in |D|}\mathbb{D}(H_{a})}^{(\mathcal{I}_{a})_{a
		\in |D|}}$ of the unfolding manifold associated to $(\mathcal{I}_{a})_{a\in |D|}$,
		each fiber $\theta_{\prod_{a\in |D|}\mathbb{D}(H_{a})}^{-1}((\iota_{\mathcal{I}_{a}}(\mathsf{c}_{a}))_{a\in |D|}))$
		is isomorphic to a Zariski open subspace of 
		the moduli space $\mathcal{M}_{s}^{*}(\mathbf{H}((\mathsf{c}_{a})_{a\in |D|})^{(\mathcal{I}_{a})_{a\in |D|}})$
		of meromorphic connections on $\mathbb{P}^{1}$.
	\end{thm}
	
	\subsubsection{Example: unfolding manifold of unramified Painlev\'e phase spaces}
	Let us consider the HTL normal form 
	\[
		H_{I\!I}:=\left(\frac{\begin{pmatrix}
			a_{3}& \\ &b_{3}
		\end{pmatrix}}{z^{3}}	
		+
		\frac{\begin{pmatrix}
			a_{2}&\\&b_{2}
		\end{pmatrix}}{z^{2}}
		+\frac{\begin{pmatrix}
			a_{1}&\\&b_{1}
		\end{pmatrix}}{z^{1}}
		+
		\begin{pmatrix}
			a_{0}&\\&b_{0}
		\end{pmatrix}			
		\right)
		\frac{dz}{z}
	\]
	Then it is known that $\mathcal{M}_{s}^{*}(H_{I\!I})$ is the phase space of 
	the Hamiltonian system of the Painlev\'e $I\!I$ equation.
	The unfolding of $H_{I\!I}$ is 
	\begin{align*}
		&H_{I\!I}(c_{0},c_{1},c_{2},c_{3})=\\
		&\left(\frac{\begin{pmatrix}
			a_{3}& \\ &b_{3}
		\end{pmatrix}}{(z-c_{1})(z-c_{2})(z-c_{3})}	
		+
		\frac{\begin{pmatrix}
			a_{2}&\\&b_{2}
		\end{pmatrix}}{(z-c_{1})(z-c_{2})}
		+\frac{\begin{pmatrix}
			a_{1}&\\&b_{1}
		\end{pmatrix}}{(z-c_{1})}
		+
		\begin{pmatrix}
			a_{0}&\\&b_{0}
		\end{pmatrix}			
		\right)
		\frac{dz}{z-c_{0}}.
	\end{align*}
	The unfolding parameter space $\mathbb{C}^{4}$ has the following $5$ types of strata
	\begin{align*}
		C_{I\!I}&:=\{(c_{i})_{i=0,\ldots,3}\in \mathbb{C}^{4}\mid c_{0}=c_{1}=c_{2}=c_{3}\},\\
		C_{I\!I\!I}^{(i,j|k,l)}&:=\{(c_{i})_{i=0,\ldots,3}\in \mathbb{C}^{4}\mid c_{i}=c_{j},\,c_{k}=c_{l},\,c_{i}\neq c_{k}\},\\
		C_{I\!V}^{(i|j,k,l)}&:=\{(c_{i})_{i=0,\ldots,3}\in \mathbb{C}^{4}\mid c_{j}=c_{k}=c_{l},\, c_{i}\neq c_{j}\},\\
		C_{V}^{(i|j|k,l)}&:=\{(c_{i})_{i=0,\ldots,3}\in \mathbb{C}^{4}\mid c_{k}=c_{l},\,c_{i}\neq c_{j},\,c_{i}\neq c_{k},\,c_{j}\neq c_{k} \},\\
		C_{V\!I}^{(i|j|k|l)}&:=\{(c_{i})_{i=0,\ldots,3}\in \mathbb{C}^{4}\mid c_{s}\neq c_{t},\,s\neq t \},
	\end{align*}
	where $\{i,j,k,l\}=\{0,1,2,3\}$.
	The diagram below is the Hasse diagram of 
	the closure relation of the strata $C_{I\!I}, C_{I\!I\!I}^{(0,1|2,3)}, C_{I\!V}^{(0|1,2,3)},
	C_{V}^{(0|1|2,3)}, C_{V\!I}^{(0|1|2|3)}$.
	\[
		\begin{tikzpicture}
			\node (s) {$C_{V\!I}^{(0|1|2|3)}$};
			\node[below =of s] (a) {$C_{V}^{(0|1|2,3)}$};
			\node[left=1cm of a] (a1) {};
			\node[right=1cm of a] (a2) {};
			\node[below=1cm of a1] (b) {$C_{I\!V}^{(0|1,2,3)}$};
			\node[below=1cm of a2] (c) {$C_{I\!I\!I}^{(0,1|2,3)}$};
			\node[below =2cm of a] (d) {$C_{I\!I}$};
			\draw[-] (s) -- (a);
			\draw[-] (a) -- (b);
			\draw[-] (a) -- (c);
			\draw[-] (c) -- (d);
			\draw[-] (b) -- (d);
		\end{tikzpicture}
	\]
	
	If $\mathbf{c}=(c_{0},\ldots,c_{3})\in C_{I\!I\!I}^{(i,j|k,l)}\cap \mathbb{D}(H_{I\!I})$ for example, 
	there exist holomorphic functions $a^{(\mu)}_{\nu}(\mathbf{c})
	b^{(\mu)}_{\nu}(\mathbf{c})$ of $\mathbf{c}\in C_{I\!I\!I}^{(i,j|k,l)}\cap \mathbb{D}(H_{I\!I})$
	such that 
	we can write 
	$H_{I\!I}(c_{0},c_{1},c_{2},c_{3})=H^{(i,j|k,l)}_{I\!I\!I,\,c_{i}}+H^{(i,j|k,l)}_{I\!I\!I,\,c_{k}}$
	by 
	\begin{align*}
		H^{(i,j|k,l)}_{I\!I\!I,\,c_{i}}&:=\left(
			\frac{\begin{pmatrix}
				a^{(i)}_{1}(\mathbf{c})&\\
				&b^{(i)}_{1}(\mathbf{c})
			\end{pmatrix}}{(z-c_{i})}
			+\begin{pmatrix}
				a^{(i)}_{0}(\mathbf{c})&\\
				&b^{(i)}_{0}(\mathbf{c})
			\end{pmatrix}
		\right)\frac{d}{z-c_{i}},\\
		H^{(i,j|k,l)}_{I\!I\!I,\,c_{k}}&:=\left(
			\frac{\begin{pmatrix}
				a^{(k)}_{1}(\mathbf{c})&\\
				&b^{(k)}_{1}(\mathbf{c})
			\end{pmatrix}}{(z-c_{k})}
			+\begin{pmatrix}
				a^{(k)}_{0}(\mathbf{c})&\\
				&b^{(k)}_{0}(\mathbf{c})
			\end{pmatrix}
		\right)\frac{d}{z-c_{k}}.
	\end{align*}
	The condition $\mathbf{c}\in C_{I\!I\!I}^{(i,j|k,l)}\cap \mathbb{D}(H_{I\!I})$
	moreover assures that the moduli space $\mathcal{M}_{s}^{*}(\mathbf{H}_{I\!I\!I})$
	associated to the collection $\mathbf{H}_{I\!I\!I}:=(H^{(i,j|k,l)}_{I\!I\!I,\,c_{i}}, H^{(i,j|k,l)}_{I\!I\!I,\,c_{k}})$
	is isomorphic to the phase space of Painlev\'e $I\!I\!I$ equation. 
	In this case, our main theorem says the following.
	\begin{prp} 
		For each $\mathbf{c}\in C_{I\!I\!I}^{(i,j|k,l)}\cap \mathbb{D}(H_{I\!I})$, the 
		fiber $\theta_{\mathbb{D}(H_{I\!I})}^{-1}(\mathbf{c})$ of 
	the unfolding manifold $\theta_{\mathbb{D}(H_{I\!I})}\colon 
	\mathcal{M}_{s}^{*}(H_{I\!I})_{\mathbb{D}(H_{I\!I})}\rightarrow \mathbb{D}(H_{I\!I})$
	is isomorphic to a Zariski open subset of $\mathcal{M}_{s}^{*}(\mathbf{H}_{I\!I\!I})$,
	the phase space of Painlev\'e $I\!I\!I$ equation.
	\end{prp}
	The same statement holds for each above  5 strata, namely, 
	we have the following.
	\begin{thm}
		The symplectic foliation of $\mathcal{M}_{s}^{*}(H_{I\!I})_{\mathbb{D}(H_{I\!I})}$
	consists of  Zariski open subsets of Painlev\'e $I\!I$ to $V\!I$ phase spaces, and 
	the types $I\!I$ to $V\!I$ of leaves $\theta_{\mathbb{D}(H_{I\!I})}^{-1}(\mathbf{c})$
	 correspond to the strata $C_{I\!I}$ to $C_{V\!I}^{(i|j|k|l)}$ containing 
	 the deformation parameter $\mathbf{c}$.
	\end{thm}

	\end{document}